\documentclass[reqno]{amsart}
\usepackage{amsfonts}
\usepackage{amssymb}
\usepackage{times}
\usepackage{xcolor}

\newcommand{\red}{\color{red}}

\newcommand{\KK}{\mathbb{K}}

\makeatletter
\def\Ddots{\mathinner{\mkern1mu\raise\p@
\vbox{\kern7\p@\hbox{.}}\mkern2mu
\raise4\p@\hbox{.}\mkern2mu\raise7\p@\hbox{.}\mkern1mu}}
\makeatother

\theoremstyle{plain}%

 \newtheorem{thm}{Theorem}[section]
 \newtheorem{cor}{Corollary}[section]
 \newtheorem{lem}{Lemma}[section]
 \newtheorem{prop}{Proposition}[section]
 \theoremstyle{definition}
 \newtheorem{defn}{Definition}[section]
 \newtheorem{exa}{Example}[section]
 \theoremstyle{remark}
 \newtheorem{rem}{Remark}[section]
 \numberwithin{equation}{section}
\newenvironment{dem}[1][Proof of Claim 1]{\noindent\textit{#1.} }{\hfill $\square$}
\newenvironment{dem2}[1][Proof of Claim 2]{\noindent\textit{#1.} }

\newfont{\hueca}{msbm10}

\begin{document}

\title{Odd-quadratic   Lie superalgebras  with  a weak filiform  module as an odd part}

\thanks{The first author was supported by the Centre for Mathematics of the University of Coimbra - UIDB/00324/2020, funded by the Portuguese Government through FCT/MCTES. Third author was supported  by Junta de Extremadura and Fondo Europeo de Desarrollo Regional (GR21005 and IB18032) and by Agencia Estatal de Investigaci\'on (Spain), grant PID2020-115155GB-I00 (European FEDER support included, UE). Fourth author was supported by the PCI of the UCA `Teor\'\i a de Lie y Teor\'\i a de Espacios de Banach' and by the PAI with project number FQM298.}

\author[Elisabete~Barreiro]{Elisabete~Barreiro}
\address{Elisabete~Barreiro,  University of Coimbra, CMUC, Department of Mathematics, Apartado 3008,
EC Santa Cruz,
3001-501 Coimbra,
Portugal\\
{\em E-mail}: {\tt mefb@mat.uc.pt}}{}

\author[Sa\"{i}d Benayadi]{Sa\"{i}d Benayadi}
\address{Sa\"{i}d Benayadi,
Laboratoire de Math\'{e}matiques IECL UMR CNRS 7502,
Universit\'{e} de Lorraine, 3 rue Augustin Fresnel, BP 45112,
F-57073 Metz Cedex 03, France\\
{\em E-mail}: {\tt said.benayadi@univ-lorraine.fr}}{}

\author[Rosa M. Navarro]{Rosa M. Navarro}
\address{Rosa M. Navarro,
Departamento de Matem{\'a}ticas, Universidad de Extremadura, C{\'a}ceres, Espa\~na\\ {\em E-mail}: {\tt rnavarro@unex.es}}{}

\author[Jos\'{e} M. S\'{a}nchez]{Jos\'{e} M. S\'{a}nchez}
\address{Jos\'{e} M. S\'{a}nchez, Departamento de Matem\'aticas, Universidad de C\'adiz, Campus de Puerto Real, 11510, Puerto Real, C\'adiz, Espa\~na\\ {\em E-mail}: {\tt txema.sanchez@uca.es}}{}

\begin{abstract}
The aim of this work is to study a very special family of odd-quadratic Lie superalgebras ${\mathfrak g}={\mathfrak g}_{\bar 0}\oplus {\mathfrak g}_{\bar 1}$ such that ${\mathfrak g}_{\bar 1}$ is a weak filiform ${\mathfrak g}_{\bar 0}$-module (weak filiform type). We introduce this concept after having proved that the unique non-zero odd-quadratic Lie superalgebra $({\mathfrak g},B)$ with ${\mathfrak g}_{\bar 1}$  a filiform ${\mathfrak g}_{\bar 0}$-module  is the abelian $2$-dimensional Lie superalgebra ${\mathfrak g}={\mathfrak g}_{\bar 0} \oplus {\mathfrak g}_{\bar 1}$ such that $\mbox{{\rm dim }}{\mathfrak g}_{\bar 0}=\mbox{{\rm dim }}{\mathfrak g}_{\bar 1}=1$. Let us note that in this context the role of the center of ${\mathfrak g}$ is crucial. Thus, we obtain an inductive description of odd-quadratic Lie superalgebras of weak filiform type via generalized odd double extensions. Moreover, we obtain the classification, up to isomorphism, for the smallest possible dimensions, that is, six and eight.


\end{abstract}

\maketitle

\section{Introduction and preliminaries}

Since arised  Lie algebras (see for instance \cite{Kacc}) they have been studied deeply. Concretely, quadratic Lie algebras have been studied in connection with many problems derived from geometry, physics and other disciplines.

Quadratic Lie algebras have been studied   deeply since they arised in connection with many problems derived from geometry, physics and other disciplines. The structure of quadratic Lie algebras plays an important role in conformal field theory and Sugawara construction exists precisely for quadratic Lie algebras \cite{11,9b}. In \cite{10,15} was introduced the idea of double extension and it allowed them to give a certain description of quadratic Lie algebras. More precisely, it is proved that every quadratic Lie algebra may be constructed as a direct sum of irreducible ones, and the latter by a sequence of double extensions. 

Let us recall that the term ``filiform" was coined by Vergne in \cite{12*} and it refers to a class of nilpotent Lie algebras, those with the longest descending central sequence. Since then, filiform Lie algebras have been largely studied, see \cite{ABG, Fia, GJM, Goze, Mill} and references therein.  Nevertheless, the concept of filiform module is more recent and was firstly introduced in  \cite{JGP} for defining filiform Lie superalgebras.

In \cite{quadraticFLSA} we studied the class of quadratic Lie superalgebras ${\mathfrak g}={\mathfrak g}_{\bar 0}\oplus {\mathfrak g}_{\bar 1}$ such that  ${\mathfrak g}_{\bar 1}$ is a filiform ${\mathfrak g}_{\bar 0}$-module (to short we  called filiform type). We showed that the study of quadratic Lie superalgebras of filiform type can be reduced to those that are solvable. Moreover, we obtained an inductive description of solvable quadratic Lie superalgebras of filiform type via both double extensions and odd double extensions of quadratic ones.

A Lie superalgebra ${\mathfrak g}$ is called {\it odd-quadratic}, on the other hand, if there is a bilinear form $B$ on ${\mathfrak g}$ such that $B$ is non-degenerate,
supersymmetric, odd and ${\mathfrak g}$-invariant. Motivated by mathematical and physical applications \cite{SV2003},
in the paper \cite{OddQuadraticLieSuperalgebras} was introduced the class of odd-quadratic Lie superalgebras studying some properties and looking at non elementary examples. In that work was given descriptions of odd-quadratic Lie superalgebras such that the even part is a reductive Lie
algebra.

It rises natural the question to  study odd-quadratic Lie superalgebras $({\mathfrak g}={\mathfrak g}_{\bar 0}\oplus {\mathfrak g}_{\bar 1},B)$   such that    ${\mathfrak g}_{\bar 1}$ has the structure of filiform ${\mathfrak g}_{\bar 0}$-module, instead of quadratic Lie superalgebras. 
We show that  the unique non-zero odd-quadratic Lie superalgebra $({\mathfrak g},B)$ verifying that ${\mathfrak g}_{\bar 1}$ has the structure of  filiform ${\mathfrak g}_{\bar 0}$-module is the $2$-dimensional Lie superalgebra ${\mathfrak g}={\mathfrak g}_{\bar 0} \oplus {\mathfrak g}_{\bar 1}$ such that $\mbox{{\rm dim }}{\mathfrak g}_{\bar 0}=\mbox{{\rm dim }}{\mathfrak g}_{\bar 1}=1$ with zero product, i.e. abelian (see Corollary \ref{Coro21}).  
We include the proof in the paper  for completeness. In this case, $({\mathfrak g}_{\bar 1})^{*}$ is a filiform ${\mathfrak g}_{\bar 0}$-module. Since ${\mathfrak g}_{\bar 0}$ is isomorphic to $({\mathfrak g}_{\bar 1})^{*}$  with isomorphism $\widetilde{B}: {\mathfrak g}_{\bar 0} \to ({\mathfrak g}_{\bar 1})^*$ defined as $\widetilde{B}(X):=B(X,\cdot),$ for $X \in {\mathfrak g}_{\bar 0}$, we conclude that ${\mathfrak g}_{\bar 0}$ is also a filiform ${\mathfrak g}_{\bar 0}$-module. Consequently, $\mbox{\rm dim } [{\mathfrak g}_{\bar 0},{\mathfrak g}_{\bar 0}] = \mbox{\rm dim } {\mathfrak g}_{\bar 0}-1$. As well, $({\mathfrak g}={\mathfrak g}_{\bar 0}\oplus {\mathfrak g}_{\bar 1},B)$ is an odd-quadratic Lie superalgebra such that  ${\mathfrak g}_{\bar 1}$ has the structure of filiform ${\mathfrak g}_{\bar 0}$-module, and the non-zero Lie algebra  $ {\mathfrak g}_{\bar 0}$ is nilpotent  verifying then  $\mbox{\rm dim } {\mathfrak g}_{\bar 0} - \mbox{\rm dim } [{\mathfrak g}_{\bar 0},{\mathfrak g}_{\bar 0}] \geq 2$, which is a contradiction. In fact, these two structures together have proven to be very restrictive.

Therefore, in this paper we present a mild version of this concept, we study odd-quadratic Lie superalgebras $({\mathfrak g} = {\mathfrak g}_{\bar 0}\oplus {\mathfrak g}_{\bar 1},B)$ such that  ${\mathfrak g}_{\bar 1}$ is a weak filiform ${\mathfrak g}_{\bar 0}$-module. In this case, we can construct a structure theory and easily present examples. Concretely, in Section 2 it is shown the non-existence of odd-quadratic Lie superalgebras $({\mathfrak g} = {\mathfrak g}_{\bar 0}\oplus {\mathfrak g}_{\bar 1},B)$ such that ${\mathfrak g}_{\bar 1}$ is a filiform ${\mathfrak g}_{\bar 0}$-module. In the next section we introduce the concept of weak filiform module (as an odd part via odd double extensions of odd-quadratic ones) and study its structure. In this part the center of ${\mathfrak g}$ plays a relevant role in its own description. Given an odd-quadratic Lie superalgebra $\displaystyle ({\mathfrak g}={\mathfrak g}_{\bar 0} \oplus {\mathfrak g}_{\bar 1},B)$ with $\mbox{\rm dim }{\mathfrak g}_{\bar 1} = m > 0$, such that ${\mathfrak g}_{\bar 1}$ is a weak filiform ${\mathfrak g}_{\bar 0}$-module, we construct an odd-quadratic Lie superalgebra $({\mathfrak t}={\mathfrak t}_{\bar 0}\oplus {\mathfrak t}_{\bar 1}, \widetilde{B})$ as the generalized odd double extension of $(\mathfrak g, B)$, by the $1$-dimensional Lie superalgebra $(\mathbb{K}e)_{\bar 1}$. In Section 4, we present  an inducctive description of odd-quadratic solvable Lie superalgebras of weak filiform type. Also we classify all the complex odd-quadratic Lie superalgebras of weak filiform type ${\mathfrak g}= {\mathfrak g}_{\bar 0} \oplus {\mathfrak g}_{\bar 1}$ for the case $\mbox{{\rm dim} }{\mathfrak g}_{\bar 0} = \mbox{{\rm dim} }{\mathfrak g}_{\bar 1} \in \{3,4\}$.

Next part of the current section is devoted to a review of the needed concepts in the sequel. For an integer $i$ we denote by $\bar{i}$ its correspondent equivalence class in $\mathbb{Z}_2 = \{\bar{0} , \bar{1}  \}$. For two integers $i, j$  we use the well-defined notation $(-1)^{\overline{ij}}$ for $(-1)^{ij} \in \{-1, 1\}$.
For a $\mathbb{Z}_2$-graded vector space $V = V_{\bar 0} \oplus V_{\bar 1}$ over the field $\mathbb{K}$, as usual, we write  $V_{\bar 0}$ for its {\em even part} and $V_{\bar 1}$ for {\em odd part}. A non-zero element $X$ of $V$ is called {\it homogeneous} if  either  $X \in V_{\bar 0}$ or $X \in V_{\bar 1}$. In this work, all elements will be supposed to be homogeneous unless otherwise indicated.
A linear map $\phi : V \to W$ between two $\mathbb{Z}_2$-graded vector spaces is called {\it even} if  $\phi(V_{\bar 0}) \subset W_{\bar 0}$ and $\phi(V_{\bar 1}) \subset W_{\bar 1}$. It is called {\it odd} if  $\phi(V_{\bar 0}) \subset W_{\bar 1}$ and $\phi(V_{\bar 1}) \subset W_{\bar 0}$. Clearly, $Hom(V,W) = Hom(V,W)_{\bar 0} \oplus Hom(V,W)_{\bar 1}$, where the first summand comprises all the even linear maps, and the second  all the odd. Tensor products $V \otimes W$ are $\mathbb{Z}_2$-graded vector spaces, where its even part is   $(V \otimes W)_{\bar 0} := (V_{\bar 0} \otimes W_{\bar 0}) \oplus (V_{\bar 1} \otimes W_{\bar 1})$ and odd part $(V \otimes W)_{\bar 1} := (V_{\bar 0} \otimes W_{\bar 1}) \oplus (V_{\bar 1} \otimes W_{\bar 0})$.

\begin{defn}
A {\it Lie superalgebra} is a $\mathbb{Z}_2$-graded vector space ${\mathfrak g} = {\mathfrak g}_{\bar 0} \oplus {\mathfrak g}_{\bar 1}$, with an even bilinear operation $[\cdot,\cdot]$, which satisfies the conditions:
\begin{itemize}
\item [i.] $[X,Y] = -(-1)^{\overline{ij}}[Y,X]$
\item [ii.] $(-1)
^{\overline{ik}} [X,[Y,Z]] + (-1)^{\overline{ij}} [Y,[Z,X]] + (-1)^{\overline{jk}} [Z, [X,Y]] = 0$ ({\it super Jacobi id.})
\end{itemize}
for any $X \in {\mathfrak g}_{\bar i}, Y \in {\mathfrak g}_{\bar j}, Z \in {\mathfrak g}_{\bar k}$, with $\bar{i},\bar{j},\bar{k} \in \mathbb{Z}_2$.
\end{defn}

All Lie superalgebras will be assumed finite-dimensional over an algebraically closed commutative field $\mathbb{K}$ of characteristic zero, unless otherwise mentioned. The general background on Lie superalgebras can be found in \cite{16}. From the previous definition ${\mathfrak g}_{\bar 0}$ is a Lie algebra, and ${\mathfrak g}_{\bar 1}$ is a ${\mathfrak g}_{\bar 0}$-module. The Lie superalgebra structure also contains the
symmetric pairing $S^2 {\mathfrak g}_{\bar 1} \to {\mathfrak g}_{\bar 0}$, which is a ${\mathfrak g}_{\bar 0}$-morphism and satisfies the super Jacobi identity applied to three
elements of ${\mathfrak g}_{\bar 1}$.


\begin{defn} \label{defB}
Let ${\mathfrak g}$ be a Lie superalgebra and let $B : {\mathfrak g} \times {\mathfrak g} \to \mathbb{K}$ be a bilinear form.
\begin{itemize}
\item[i.] $B$ is {\it supersymmetrical} if $B(X,Y) = (-1)^{\overline{ij}}B(Y,X)$, for any $X \in {\mathfrak g}_{\bar i}, Y \in {\mathfrak g}_{\bar j}$, with $\bar{i},\bar{j}  \in \mathbb{Z}_2$.
\item[ii.] $B$ is {\it skew-supersymmetrical} if $B(X,Y) = -(-1)^{\overline{ij}}B(Y,X)$, for any $X \in {\mathfrak g}_{\bar i}, Y \in {\mathfrak g}_{\bar j}$, with $\bar{i},\bar{j}  \in \mathbb{Z}_2$.
\item[iii.] $B$ is {\it invariant} if $B([X,Y],Z) = B(X,[Y,Z])$, for all $X,Y,Z \in {\mathfrak g}$.
\item[iv.] $B$ is  {\it even} if $B(X,Y) = 0$, for any $X \in {\mathfrak g}_{\bar 0}, Y \in {\mathfrak g}_{\bar 1}$.
\item[v.] $B$ is  {\it odd} if $B(X,Y) = B(Y,X) = 0$, for any $X,\ Y$ such that either $X , \ Y\in {\mathfrak g}_{\bar 0}$ or $X, \ Y \in {\mathfrak g}_{\bar 1}$.
\item [ vi.]  $B$ is  \textit{non-degenerate} if $\displaystyle X \in {\mathfrak g}$ satisfies $\displaystyle B(X,Y)=0$ for all $\displaystyle Y \in {\mathfrak g}$, then $\displaystyle X=0$. Otherwise, $B$ is called {\it degenerate}.
\end{itemize}
\end{defn}


\begin{defn}
A Lie superalgebra ${\mathfrak g}$ is {\it odd-quadratic} if there exists a bilinear form $B : {\mathfrak g} \times {\mathfrak g} \to {\mathfrak g}$ such that $B$ is supersymmetrical, invariant, odd and non-degenerate. It is denoted by $({\mathfrak g},B)$ and  $B$ is an {\it odd-invariant scalar product} on ${\mathfrak g}$. 
\end{defn}

\begin{defn}
Let $({\mathfrak g},B)$ be an odd-quadratic Lie superalgebra.
\begin{itemize}
\item[i.] A graded ideal $I$ of ${\mathfrak g}$ is {\it non-degenerate} (resp. {\it degenerate}) if the restriction of $B$ to $I \times I$ is a non-degenerate (resp. degenerate) bilinear form.
\item[ii.] $({\mathfrak g},B)$ is {\it B-irreducible} if ${\mathfrak g}$ does not have non-zero non-degenerate graded ideals.
\item[iii.] A graded ideal $I$ of ${\mathfrak g}$ is called {\it B-irreducible} if $I$ is non-degenerate and $I$ contains no non-zero non-degenerate graded ideals of ${\mathfrak g}$.
\end{itemize}
\end{defn}

\begin{defn} Let ${\mathfrak g}$ be a Lie superalgebra and let $I \subset {\mathfrak g}$ be a graded ideal.
\begin{itemize}
\item[i.] $I$ is {\it minimal} if $I \notin \{\{0\}, {\mathfrak g}\}$ and if $J$ is a graded ideal of ${\mathfrak g}$ such that $J \subset I$ then $J \in \{\{0\}, I\}$.
\item[ii.] $I$ is called {\it maximal} if $I \notin \{\{0\}, {\mathfrak g}\}$ and if $J$ is a graded ideal of ${\mathfrak g}$ such that $I \subset J$ then $J \in \{I, {\mathfrak g}\}$. 
\end{itemize} 
\end{defn}

\begin{defn} Let $({\mathfrak g},B)$ be an odd-quadratic Lie superalgebra. and let $I \subset {\mathfrak g}$ be a graded ideal. We call the {\it orthogonal of $I$ with respect to $B$} to the set $$I^{\perp} := \{X\in {\mathfrak g} : B(X,Y)=0, \hspace{0.1cm} \mbox{\rm for all } Y \in I\}.$$
\noindent We also say $I$ is {\it isotropic} if $I \subset I^{\perp}$.
\end{defn}

\begin{prop}\label{Prop 2.1} \cite{OddQuadraticLieSuperalgebras}
The Lie superalgebra ${\mathfrak g}={\mathfrak g}_{\bar 0} \oplus {\mathfrak g}_{\bar 1}$ is odd-quadratic if and only if there exists an isomorphism of ${\mathfrak g}_{\bar 0}$-modules $\varphi: {\mathfrak g}_{\bar 0} \to ({\mathfrak g}_{\bar 1})^{*}$ such that $$\varphi([X,Y])(Z)=\varphi([Y,Z])(X),$$ for all $X,Y,Z \in {\mathfrak g}_{\bar 1}$. In this case, $ \mbox{\rm dim } {\mathfrak g}_{\bar 0} = \mbox{\rm dim } {\mathfrak g}_{\bar 1}$ and dimension of ${\mathfrak g}$ is even.
\end{prop}

\begin{defn}
Given a Lie superalgebra ${\mathfrak g}$, a $\mathbb{Z}_2$-graded vector space $A = A_{\bar 0} \oplus A_{\bar 1}$ is a $\displaystyle {\mathfrak g}$-\textit{module} if $A$ is equipped with an even bilinear map ${\mathfrak g} \times A \to A$ (denoted by $\displaystyle
(X,a)\mapsto Xa$, for $\displaystyle X \in {\mathfrak g}$ and
$\displaystyle a\in A$)  satisfying $$[X,Y]a = X(Ya) - (-1)^{\overline{ij}}Y(Xa),$$ for any $a \in A, X \in {\mathfrak g}_{\bar i}, Y \in {\mathfrak g}_{\bar j} $, with $\bar{i},\bar{j}  \in \mathbb{Z}_2$.
\end{defn}

 Occasionally, we have to change the gradation of Lie superalgebras
as we will describe. Let $\displaystyle {\mathfrak h}={\mathfrak h}_{\bar
0}\oplus{\mathfrak h}_{\bar 1}$ be a Lie superalgebra. Denote by
$\displaystyle P({\mathfrak h})=V_{\bar 0}\oplus V_{\bar 1}$ the $\displaystyle \mathbb{Z}_2$-graded vector space obtained from
$\displaystyle {\mathfrak h}$ with gradation defined by
\begin{eqnarray*}
\displaystyle V_{\bar 0}={\mathfrak h}_{\bar 1} & \mbox{ and }&
V_{\bar 1}={\mathfrak h}_{\bar 0}.
\end{eqnarray*}
Clearly, the associative superalgebras $\displaystyle
\mbox{Hom}({\mathfrak h})$ and $\displaystyle \mbox{Hom}(P({\mathfrak
h}))$ coincide. Therefore the representations of a Lie
superalgebra $\displaystyle {\mathfrak g}$ in $\displaystyle {\mathfrak
h}$ coincide with representations of $\displaystyle {\mathfrak g}$ in $\displaystyle P({\mathfrak h})$. Note that the dual spaces
$P({\mathfrak h}^*), {\mathfrak h}^*$ are equal as $\displaystyle \mathbb{Z}_2$-graded vector
spaces, however
\begin{eqnarray*}
\displaystyle V_{\bar 0}^*={\mathfrak h}_{\bar 1}^* & \mbox{ and }& V_{\bar 1}^*={\mathfrak h}_{\bar 0}^*.
\end{eqnarray*}
Denote by $\displaystyle \pi_{\mathfrak h} : {\mathfrak h} \to \mbox{Hom}( P({\mathfrak h}^*))$ the linear map defined for homogeneous elements as follows $$\pi_{\mathfrak h} (Z)(f)(Y) := -(-1)^{{\bar k}{\overline \delta}}f([Z,Y])$$ for all $f \in (P({\mathfrak h}^*))_{{\overline \delta}}$, $Z \in {\mathfrak h}_{\bar k}, Y \in {\mathfrak h}$.
It is quite easy to show that $\displaystyle \pi_{\mathfrak h}$ is a representation of $\displaystyle {\mathfrak h}$ in $\displaystyle P({\mathfrak h}^*)$, but it is not the co-adjoint representation of
$\displaystyle {\mathfrak h}$.

\begin{prop}\label{centralext}\cite{OddQuadraticLieSuperalgebras}
Let $({\mathfrak g},B)$ be an odd-quadratic Lie superalgebra and ${\mathfrak h}$ a Lie superalgebra. Consider $\psi : {\mathfrak h} \to Der_a ({\mathfrak g},B)$ a morphism of Lie superalgebras. Define the linear map $\varphi : {\mathfrak g} \times {\mathfrak g} \to P({\mathfrak h}^*)$ for homogeneous elements by  $$\varphi(X,Y)(Z) := (-1)^{({\bar i}+{\bar j}) {\bar k}} B \bigl(\psi(Z)(X),Y\bigr)$$ for all $X \in {\mathfrak g}_{\bar i}, Y \in {\mathfrak g}_{\bar j}, Z \in {\mathfrak h}_{\bar k}$. Then the vector space ${\mathfrak g} \oplus P({\mathfrak h}^*)$ becomes a Lie superalgebra considering the product $$[X+f,Y+h] := [X,Y]_{\mathfrak g} + \varphi(X,Y),$$
for $X+f, Y+h \in {\mathfrak g} \oplus P({\mathfrak h}^*)$. It is called the {\it central extension of $ P({\mathfrak h}^*)$ by ${\mathfrak g}$} (by means of $\varphi$).
\end{prop}

\begin{thm}
\cite{OddQuadraticLieSuperalgebras}
Let $({\mathfrak g},B)$ be an odd-quadratic Lie superalgebra and ${\mathfrak h}$ a Lie superalgebra. Consider $\psi : {\mathfrak h} \to Der_a({\mathfrak g},B)$ a morphism of Lie superalgebras. Define the linear map $ \widetilde{\psi}: {\mathfrak h} \to \mbox{\rm Hom}({\mathfrak g} \oplus P({\mathfrak h}^*))$ by $$\widetilde{\psi} (Z)(X+f) := \psi (Z)(X)+ \pi _{\mathfrak h}(Z)(f)$$
for $X+f \in {\mathfrak g} \oplus P({\mathfrak h}^*), Z \in {\mathfrak h}$.

\noindent Then $\widetilde{\psi}(Z) \in Der ({\mathfrak g} \oplus P({\mathfrak h}^*))$, for $Z \in {\mathfrak h}_{\bar k}$ where ${\mathfrak g} \oplus P({\mathfrak h}^*)$ is the central extension of $P({\mathfrak h}^*)$ by ${\mathfrak g}$ (by means of $\varphi$). Moreover, ${\mathfrak t} = {\mathfrak h}\oplus {\mathfrak g} \oplus P({\mathfrak h}^*)$ with the product
\begin{align*}
[Z+X+f,W+Y+h] := & [Z,W]_{\mathfrak h}+[X,Y]_{\mathfrak g}+\psi(Z)(Y)- (-1)^{{\bar i}{\bar j}}\psi(W)(X)\\ & + \pi_{\mathfrak h}(Z)(\mathfrak h) -(-1)^{{\bar i} {\bar j}}\pi_{\mathfrak h}(W)(f)+\varphi(X,Y),
\end{align*}
with $Z+X+f \in {\mathfrak t}_{{\bar i}}, W+Y+h \in {\mathfrak t}_{\bar j}$, where $\varphi$ is defined in Proposition \ref{centralext}, is a Lie superalgebra. More precisely, ${\mathfrak t}$ is the semi-direct product of ${\mathfrak g} \oplus P({\mathfrak h}^*)$ by ${\mathfrak h}$ by means of $\widetilde{\psi}$. Furthermore, let $\gamma$ be an odd supersymmetric invariant bilinear form on ${\mathfrak h}$ (not necessarily non-degenerate). Then the bilinear form $\widetilde{B}: {\mathfrak t} \times {\mathfrak t} \to \mathbb{K}$ defined by $$\widetilde{B} (Z+X+f,W+Y+h) := B(X,Y) + \gamma(Z,W) + f(W) + (-1)^{{\bar i}{\bar j}}h(Z)$$ whenever $Z+X+f \in {\mathfrak t}_{\bar i}, W+Y+h \in {\mathfrak t}_{\bar j}$, is an odd-invariant scalar product on ${\mathfrak t}$ and $({\mathfrak t}, \widetilde{B})$ is an odd-quadratic Lie superalgebra. We say that the odd-quadratic Lie superalgebra $({\mathfrak t},\widetilde{B})$ is an {\it odd double extension of $({\mathfrak g},B)$ by} ${\mathfrak h}$ (by means of $\psi$ and $\gamma$).
\end{thm}

\section{Non-existence of odd-quadratic Lie superalgebras ${\mathfrak g}={\mathfrak g}_{\bar 0}\oplus {\mathfrak g}_{\bar 1}$ such that ${\mathfrak g}_{\bar 1}$ is a filiform ${\mathfrak g}_{\bar 0}$-module and $\mbox{{\rm dim }} {\mathfrak g} >2$}
Our first attempt was to study odd-quadratic Lie superalgebras ${\mathfrak g}={\mathfrak g}_{\bar 0}\oplus {\mathfrak g}_{\bar 1}$ such that ${\mathfrak g}_{\bar 1}$ is a filiform ${\mathfrak g}_{\bar 0}$-module, but these two structures, odd quadratic and filiform module,  are shown to be  incompatible together if $\mbox{{\rm dim }} {\mathfrak g} >2$. 
We  explain the idea of the proof. We show that  the dual  ${\mathfrak g}^{*}_{\bar 1}$ is also a filiform ${\mathfrak g}_{\bar 0}$-module. As ${\mathfrak g}_0$-modules ${\mathfrak g}_0$ and ${\mathfrak g}_1^*$ are isomorphic, then ${\mathfrak g}_0$ is a filiform ${\mathfrak g}_0$-module. This fact leads, in particular, to $\mbox{{\rm dim }}[{\mathfrak g}_{\bar 0}, {\mathfrak g}_{\bar 0}] = \mbox{{\rm dim }} {\mathfrak g}_{\bar 0}-1$. We   prove also that ${\mathfrak g}_{\bar 0}$ is nilpotent  Lie algebra, so it  has at least two generators and   $\mbox{{\rm dim }}[{\mathfrak g}_{\bar 0}, {\mathfrak g}_{\bar 0}] \leq \mbox{{\rm dim }}{\mathfrak g}_{\bar 0}-2$, arriving then to a contradiction.

We recall the  concept of filiform module that was firstly introduced in \cite{JGP} for defining filiform Lie superalgebras.

\begin{defn}\rm
Let $(\mathfrak{g},[\cdot,\cdot])$ be a Lie algebra and $(V,\pi)$ be a finite dimensional representation of $\mathfrak{g}$ (that is, $V$ is a $\mathfrak{g}$-module). We say $V$ is a {\it filiform $\mathfrak{g}$-module} if there exist $\{V_0,\dots,V_m\}$, being $m = \mbox{\rm dim }V$, such that
\begin{itemize}
\item[1.] $V_i$ is a $\mathfrak{g}$-submodule of $V$, for  $i \in \{0,\dots,m\}$.
\item[2.] $\mbox{\rm dim }V_i = i,$ for $i \in \{0,\dots,m\}$.
\item[3.] $V=V_m \supset V_{p-1} \supset \cdots \supset V_1 \supset V_0 = \{0\}$.
\item[4.] $\mathfrak{g} \cdot V_i = V_{i-1}$, for $i \in \{1,\dots,m\}$.
\end{itemize}
We recall that $X\cdot v=\pi(X)(v)$ for any $X \in \mathfrak{g}, v \in V$.
\end{defn}

\noindent We need the following three lemmas to prove the main theorem of the section.  

\begin{lem}\label{lemmma_2}
Let $(\mathfrak{g},[\cdot,\cdot])$ be a Lie algebra and $(V,\pi)$ a filiform representation (i.e. $V$ is a filiform $\mathfrak{g}$-module). Then the dual representation $(V^*,\pi^*)$ is filiform (i.e. $V^*$ is a filiform $\mathfrak{g}$-module).
\end{lem}
\begin{proof}
Recall that  $X\cdot f = \pi^*(X)f = -f \circ \pi(X)$ for all $X \in \mathfrak{g}, f \in V^*$. Then  $$\bigr((\pi^*)(X)(f)\bigl)(v)=-f(\pi(X)(v)) = (X\cdot f)(v)=-f(X\cdot v)$$ for any $v \in V$.
Let us consider a basis $\mathcal{B}:=\{e_1,\dots,e_m\}$ of $V$ such that $\{e_1,\dots,e_i\}$ is a basis of $V_i$, for $i \in \{1,\dots,m\}$. For any $i \in \{0,\dots,m\}$ we define $$(V^*)_i := \{f \in V^* : f(V_{m-i})=\{0\}\}.$$ It is clear that $(V^*)_i$ is a vector subspace of $V^*$ and $\{e_{m-(i-1)}^*,\dots,e_m^*\}$ is a basis of $(V^*)_i$, for $i \in \{1,\dots,m\}$, where $\{e_1^*,\dots,e_m^*\}$ is a dual basis of $\mathcal{B}$. So $\mbox{\rm dim }(V^*)_i = i$.

For $i \in \{1,\dots,m\}$,
\begin{align*}
f \in (V^*)_i &\iff f(V_{m-i})=\{0\}\\
&\iff f(X\cdot v)=\{0\} \; \mbox{\rm for all } X \in \mathfrak{g}, v \in V_{m-i+1}\\
&\iff (X\cdot f)(V_{m-(i-1)})=\{0\} \; \mbox{\rm for all } X \in \mathfrak{g}\\
&\iff X\cdot f \in (V^*)_{i-1} \; \mbox{\rm for all } X \in \mathfrak{g}
\end{align*}
We conclude that $\mathfrak{g} \cdot (V^*)_i \subset (V^*)_{i-1}$.

If $i=1,$  $(V^*)_1= {\rm span}_{\mathbb K}\{e_m^*\}$, then $\mathfrak{g} \cdot (V^*)_1=(V^*)_0 = \{0\}$. Let $i \in \{2,\dots,m\}$, in the following we are going to prove that $\mathfrak{g} \cdot (V^*)_i=(V^*)_{i-1}$, with $(V^*)_{i-1}={\rm span}_{\mathbb K}\{e_{m-i+2}^*,\dots,e_m^*\}$.
Let $q \in \{m-i+2,\dots,m\} \subset \{2,\dots,m\}$. Then there exists $X_q \in \mathfrak{g}$ such that $$X_qe_q=\beta_{q-1}e_{q-1}+\cdots+\beta_1e_1$$ with $\beta_{q-1} \neq 0$ because $\mathfrak{g}V_q = V_{q-1}$.

Now,
$$X_q((-\beta_{q-1}^{-1})e_{q-1}^*)(e_q) =\beta_{q-1}^{-1}e_{q-2}^*(Xe_q)
=\beta_{q-1}^{-1}\beta_{q-1}
=1.$$
Therefore, $X_q((\beta_{q-1}^{-1})e_{q-1}^*) = e_q^*+X_q^*$ where $X_q^* \in {\rm span}_{\mathbb K}\{e_{q+1}^*,\dots,e_m^*\}$ if and only if $X_q(e_{q-1}^*) =\lambda_qe_q^*+f_q$, where $\lambda_q := \beta_{q-1}, f_q:=\beta_{q-1}X_q^*$.
In conclusion, for $q \in \{m-i+2,\dots,m-1\}$ there exist $X_q \in \mathfrak{g}$, $\lambda_q \in \mathbb{K} \setminus \{0\}$ and $f_q \in {\rm span}_{\mathbb{K}}\{e_{q+1}^*,\dots,e_m^*\}$ such that $$X_qe_{q-2}^* = \lambda_qe_q^*+f_q.$$
For $q=m$ there exist $X_m \in \mathfrak{g}$, $\lambda_m \in \mathbb{K}\setminus \{0\}$ such that $$X_me_{m-1}^*=\lambda_me_m^*.$$
It follows that $$\{e_{m-i+2}^*,\dots,e_m^*\} \subset \mathfrak{g} \cdot (V^*)_i.$$ Then $(V^*)_{i-1} \subset \mathfrak{g} \cdot (V^*)_i$ and we conclude that $\mathfrak{g} \cdot (V^*)_i=(V^*)_{i-1}$ which proves that $V^*$ is a filiform $\mathfrak{g}$-module.
\end{proof}

\begin{lem}\label{lemmma_3}
Let $(\mathfrak{g} = \mathfrak{g}_{\bar 0} \oplus \mathfrak{g}_{\bar 1},[\cdot,\cdot])$ be a Lie superalgebra. If $\mathfrak{g}$ admits an odd symmetric non-degenerate and invariant bilinar form $B$, then the $\mathfrak{g}_{\bar 0}$-module $\mathfrak{g}_{\bar 0}$ and the dual of the $\mathfrak{g}_{\bar 0}$-module $\mathfrak{g}_{\bar 1}$ are isomorphic as $\mathfrak{g}_{\bar 0}$-modules.
\end{lem}
\begin{proof}
We define $\Phi : \mathfrak{g}_{\bar 0} \to (\mathfrak{g}_{\bar 1})^*$ as $\Phi(X):=B(X,\cdot)$ for all $X \in \mathfrak{g}_{\bar 0}$. It is clear that $\Phi$ is an isomorphism of vector spaces.
For $X,A \in \mathfrak{g}_{\bar 0}, X'\in \mathfrak{g}_{\bar 1}$ we obtain
\begin{align*}
\Phi(A \cdot X)(X')&=B(A \cdot X,\cdot)(X')\\
&=B(A \cdot X,X')\\
&=B([A,X],X')\\
&=-B(X,[A,X'])\\
&=-\Phi(X)([A,X'])\\
&=(A\cdot\Phi(X))(X').
\end{align*}
Consequently $\Phi(A \cdot X)=A\cdot\Phi(X)$, which proves that $\Phi$ is an isomorphisms of $\mathfrak{g}_{\bar 0}$-modules.
\end{proof}

\begin{lem}\label{lemmma_1}
If $(\mathfrak{g},[\cdot,\cdot])$ is a nilpotent Lie algebra such that $\mbox{\rm dim }\mathfrak{g} \geq 2$, then   $\mbox{\rm dim }\mathfrak{g} - \mbox{\rm dim }[\mathfrak{g},\mathfrak{g}] \geq 2$.
\end{lem}
\begin{proof}
Let $(\mathfrak{g},[\cdot,\cdot])$ be a nilpotent Lie algebra verifying $\mbox{\rm dim }\mathfrak{g} = m \geq 2$. By Engel's Theorem, there exists a basis $\mathcal{B} := \{e_1,\dots,e_m\}$ of $\mathfrak{g}$ such that: for all $X \in \mathfrak{g},$
\begin{itemize}
\item[1.] $[X,e_1]=0$, 
\item[2.] $[X,e_i]=\sum_{j=1}^{i-1}a_{ji}(X)e_j$, for any $i \in \{2,\dots,m\}$.
\end{itemize}
Consequently, $[e_1,\mathfrak{g}] = \{0\}$ and  for any $i \in \{2,\dots,m-1\}, j \in \{i+1,\dots,m\}$, $$[e_i,e_j] \in {\rm span}_{\mathbb K}\{e_1,\dots,e_{i-1}\}.$$
We conclude that $[\mathfrak{g},\mathfrak{g}] \subset {\rm span}_{\mathbb K}\{e_1,\dots,e_{m-2}\}$.
It follows that $\mbox{\rm dim }[\mathfrak{g},\mathfrak{g}] \leq m-2$, this is $\mbox{\rm dim }\mathfrak{g} - \mbox{\rm dim }[\mathfrak{g},\mathfrak{g}] \geq 2$ as  required.
\end{proof}

\begin{thm}
 There are no odd-quadratic Lie superalgebra $(\mathfrak{g},[\cdot,\cdot],B)$ such that the $\mathfrak{g}_{\bar 0}$-module $\mathfrak{g}_{\bar 1}$ is filiform and $\mbox{\rm dim }(\mathfrak{g}_{\bar 0})\geq 2$.
 \end{thm}
\begin{proof} Suppose that $(\mathfrak{g},[\cdot,\cdot],B)$ is an odd-quadratic Lie superalgebra such that the $\mathfrak{g}_{\bar 0}$-module $\mathfrak{g}_{\bar 1}$ is filiform and $\mbox{\rm dim }(\mathfrak{g}_{\bar 0})\geq 2$.
By Lemma \ref{lemmma_2}, the $\mathfrak{g}_{\bar 0}$-module $(\mathfrak{g}_{\bar 1})^*$ is filiform. It follows, by Lemma \ref{lemmma_3} that the $\mathfrak{g}_{\bar 0}$-module $\mathfrak{g}_{\bar 0}$ is filiform. Therefore there exists $\{I_0,\dots,I_m\}$, being $m=\mbox{\rm dim }\mathfrak{g}_{\bar 0}$, satisfying
\begin{itemize}
\item[1.] $I_i$ is an ideal of $\mathfrak{g}_{\bar 0}$ with $\mbox{\rm dim }(I_i)=i$ for $i \in \{0,\dots,m\}$.
\item[2.] $\mathfrak{g}_{\bar 0}=I_m \supset I_{m-1} \supset \cdots \supset I_1 \supset I_0 = \{0\}$.
\item[3.] $[\mathfrak{g}_{\bar 0},I_i]=I_{i-1}$, for $i \in \{1,\dots,m\}$.
\end{itemize}
This implies that $(\mathfrak{g}_{\bar 0},[\cdot,\cdot]|_{\mathfrak{g}_{\bar 0} \times \mathfrak{g}_{\bar 0}})$ is a $n$-dimensional nilpotent Lie algebra with $n \geq 2$, and $[\mathfrak{g}_{\bar 0},\mathfrak{g}_{\bar 0}] = I_{m-1}$. So $$\mbox{\rm dim }[\mathfrak{g}_{\bar 0},\mathfrak{g}_{\bar 0}] = (\mbox{\rm dim }\mathfrak{g}_{\bar 0})-1$$ which contradicts Lemma \ref{lemmma_1}. This completes the proof.
\end{proof}

\begin{cor}\label{Coro21}
The unique non-zero odd-quadratic Lie superalgebra $(\mathfrak{g},[\cdot,\cdot],B)$ such that the $\mathfrak{g}_{\bar 0}$-module $\mathfrak{g}_{\bar 1}$ is  filiform is the $2$-dimensional Lie superalgebra $\mathfrak{g}=\mathfrak{g}_{\bar 0}\oplus \mathfrak{g}_{\bar 1}$ such that $\mbox{\rm dim }\mathfrak{g}_{\bar 0} = \mbox{\rm dim }\mathfrak{g}_{\bar 1}=1$ with zero product.
\end{cor}

\section{Lie superalgebras with a weak filiform module as an odd part via odd double extensions of odd-quadratic ones}

In this work  we consider a  less restrictive concept for our purpose, as  we can construct a structure  theory and  present  examples, which we name by  weak filiform. Firstly, we introduce the  concept of weak filiform module in general context.

\begin{defn}
Let $\mathfrak{g}$ be a Lie algebra and  $V$ be a vector space with $\mbox{\rm dim }V = m > 0$. We say that the ${\mathfrak g}$-module $V$ is {\it weak filiform} if the action of ${\mathfrak g}$ on $V$ defines a {\it flag}, in the sense  that there exists a decreasing subsequence of vector subspaces in its underlying vector space $V$, $$V = V_m \supset V_{m-1} \supset \dots \supset V_2 \supset V_1 = \{0\},$$ with dimensions $\mbox{\rm dim } V_i = i$, for $i \in \{2,\dots,m\}$ and such that $[{\mathfrak g},V_i]= V_{i-1}$ for any $i \in \{2,\dots,m\}$. 
\end{defn}

\noindent Now we present the notion of weak filiform Lie superalgebra.

\begin{defn}
Let ${\mathfrak g} = {\mathfrak g}_{\bar 0} \oplus {\mathfrak g}_{\bar 1}$ be a Lie superalgebra with $\mbox{\rm dim } {\mathfrak g}_{\bar 1} =m  > 0$. We say that ${\mathfrak g}_{\bar 1}$ has the structure of {\it weak  filiform ${\mathfrak g}_{\bar 0}$-module} if the action of ${\mathfrak g}_{\bar 0}$ on ${\mathfrak g}_{\bar 1} $ defines a {\it flag}, that is, a decreasing subsequence of vector subspaces in its underlying vector space ${\mathfrak g}_{\bar 1}$, $${\mathfrak g}_{\bar 1} = V_m \supset V_{m-1} \supset \dots \supset V_2 \supset V_1 = \{0\},$$ with dimensions $\mbox{\rm dim } V_i = i$ for any $i \in \{2,\dots,m\}$, such that $[{\mathfrak g}_{\bar 0},V_i]= V_{i-1}$ for $i \in \{2,\dots,m\}$. 
To abbreviate, in the sequel we will refer to ${\mathfrak g} = {\mathfrak g}_{\bar 0} \oplus {\mathfrak g}_{\bar 1}$ as a Lie superalgebra of \textit{ weak filiform type}. We set $V_2 := \mathbb{K}u_2 \oplus \mathbb{K}v_2$ and  $V_i/V_{i-1}:=\mathbb{K}e_i$ for $i \in \{3,\dots,m\}$.
\end{defn}

\begin{exa}
Let ${\mathfrak h}$ be the Lie algebra generated by the basis $\{X_1,\dots, X_m\}$ with product $[X_1,X_i]=X_{i+1}$, for $i \in \{2,\dots,m-1\}$. Consider the Lie superalgebra ${\mathfrak g} = {\mathfrak h} \oplus P({\mathfrak h}^*)$, as usual $ {\mathfrak h}^*$ is the dual of $  {\mathfrak h}$. Then a basis of ${\mathfrak g}$ is $\{X_1,\dots,X_m,X_1^*,\dots,X_m^*\},$ being $\{X_1^*,\dots,X_m^*\}$ a dual basis of $P({\mathfrak h}^*)$, and we get  ${\mathfrak g}_{\bar 0}= {\mathfrak h}$ and ${\mathfrak g}_{\bar 1} = P({\mathfrak h}^*).$
Let the product in ${\mathfrak g}$ be defined by $[X_i,X_j^*]=X_i \cdot X_j^* = - X_j^* \circ {\rm ad}_{\mathfrak h}(X_i)$ and $[X_i^*,X_j^*]=0$, for $i,j \in \{1,\dots,m\}$. We have that $X_m, X_1^*,X_2^*$ are the central elements and   non-zero products in ${\mathfrak g}$ are
$$\begin{cases}
[X_1,X_i] = X_{i+1} \mbox{ for } i \in \{2,\dots,m-1\} \\
[X_1,X_i^*]=-X_{i-1}^* \mbox{ for } i \in \{3,\dots,m\}  \\ 
[X_i,X_{i+1}^*]=X_1^* \mbox{ for } i \in \{2,\dots,m-1\} \end{cases}$$
We define $B(X,F):=F(X)$ for any $X \in \mathfrak{h}$ and $F \in \mathfrak{h}^*$. Then $({\mathfrak g} = {\mathfrak h} \oplus P({\mathfrak h}^*),B)$ is an odd-quadratic Lie superalgebra such that the ${\mathfrak h}$-module $P({\mathfrak h}^*)$ is weak filiform with flag $$P({\mathfrak h}^*) = V_m \supset V_{m-1} \supset \dots \supset V_2
\supset V_1 = \{0\}$$ where each $V_i$ has the basis $\{X_1^*,\dots,X_i^*\}$ for $i \in \{2,\dots,m\}$.
\end{exa}

\begin{exa}
Let $m \in \mathbb{N} \setminus\{0,1,2\}$. Let us consider the filiform Lie algebra $(\tilde{\mathfrak{g}}_m,[\cdot,\cdot]_m)$ where $\tilde{\mathfrak{g}}_m:={\rm span}_{\mathbb K}\{e_1,\dots,e_m\}$ and $[e_1,e_i]_m:=e_{i-1}$ for $i \in \{3,\dots,m\}$.

Now, we consider the $\mathbb{Z}_2$-graded vector space $\mathfrak{h}=\mathfrak{h}_{\bar 0}\oplus \mathfrak{h}_{\bar 1}$, where $\mathfrak{h}_{\bar 0}:=\tilde{\mathfrak{g}}_m$ and $\mathfrak{h}_{\bar 1}:=(\tilde{\mathfrak{g}}_m)^*$, the dual of the underlying vector space of $\tilde{\mathfrak{g}}_m$.

The product $$[X+f, Y+h]:=[X,Y]_m-h \circ {\rm ad}_X + f\circ {\rm ad}_Y,$$ for $X,Y \in \mathfrak{h}_{\bar 0}, f,h \in \mathfrak{h}_{\bar 1}$, where ${\rm ad}_X:=[X,\cdot]_m$, defines a Lie superalgebra structure on $\mathfrak{h}$.

Let $B : \mathfrak{h} \times \mathfrak{h} \to \mathbb{K}$ the bilinear form given as
$$B(X,f)=B(f,X):=f(X), \hspace{2cm} B(X,Y)=B(f,h):=0$$ for $X,Y \in \mathfrak{h}_{\bar 0}, f,h \in \mathfrak{h}_{\bar 1}$. We get $B$ is odd symmetric invariant and non-degenerate. Then $(\mathfrak{h},[\cdot,\cdot],B)$ is an odd quadratic Lie superalgebra. We can easily check that the $\mathfrak{h}_{\bar 0}$-module $\mathfrak{h}_{\bar 1}$ is a weak filiform $\mathfrak{h}_{\bar 0}$-module.
\end{exa}

Next we study the structure of odd-quadratic Lie superalgebras $({\mathfrak g}={\mathfrak g}_{\bar 0}\oplus {\mathfrak g}_{\bar 1},B)$  $({\mathfrak g}={\mathfrak g}_{\bar 0}\oplus {\mathfrak g}_{\bar 1},B)$ such that the ${\mathfrak g}_{\bar 0}$-module ${\mathfrak g}_{\bar 1}$ has the structure of weak filiform ${\mathfrak g}_{\bar 0}$-module.

\begin{prop}\label{lemgo}
Let $({\mathfrak g}={\mathfrak g}_{\bar 0}\oplus {\mathfrak g}_{\bar 1},B)$ be an odd-quadratic Lie superalgebra such that the ${\mathfrak g}_{\bar 0}$-module ${\mathfrak g}_{\bar 1}$ has the structure of weak filiform ${\mathfrak g}_{\bar 0}$-module. Then the Lie algebra  $ {\mathfrak g}_{\bar 0}$ is nilpotent.
\end{prop}

\begin{proof} 
Since $({\mathfrak g},B)$ is an odd-quadratic Lie superalgebra, then there exists an isomorphism of ${\mathfrak g}_{\bar 0}$-modules  $\varphi: {\mathfrak g}_{\bar 0} \to ({\mathfrak g}_{\bar 1})^{*}$ given by $\varphi(X)=B( X,.),$ for all $X  \in {\mathfrak g}_{\bar 0}$ (see Proposition \ref{Prop 2.1}). If, in addition, the ${\mathfrak g}_{\bar 0}$-module ${\mathfrak g}_{\bar 1}$ has the structure of weak filiform ${\mathfrak g}_{\bar 0}$-module, then the representation $\pi : {\mathfrak g}_{\bar 0} \to \mathfrak{gl}({\mathfrak g}_{\bar 1})$ defined by $\pi(x):=(\mbox{\rm ad}_{\mathfrak g}(x))|_{{\mathfrak g}_{\bar 1}}$ for $x \in {\mathfrak g}_{\bar 0}$, satisfies
there exists a non-zero  $ n \in \mathbb{N}$   such that for any $x \in {\mathfrak g}_{\bar 0} $   verifies   $\bigl(\pi(x)\bigr)^n=0. $

Let us consider the dual representation $\pi^* : {\mathfrak g}_{\bar 0} \to \mathfrak{gl}({\mathfrak g}_{\bar 1}^*)$. For any $x \in {\mathfrak g}_{\bar 0}$ it maps $\pi^*(x) : {\mathfrak g}_{\bar 1}^* \to {\mathfrak g}_{\bar 1}^*$ which assigns for any $f \in {\mathfrak g}_{\bar 1}^*$, $$ \pi^*(x)(f)=-f \circ \pi(x).$$ It follows that the dual representation $\pi^*  $ also verifies that
there exists a non-zero $ n \in \mathbb{N} $ such that for any $x \in {\mathfrak g}_{\bar 0}$ verifies   $\bigl(\pi^*(x)\bigr)^n=0.$

Therefore, by the isomorphism of Proposition \ref{Prop 2.1}, there exists a non-zero $n \in \mathbb{N} $ such that $\bigl(\mbox{\rm ad}_{{\mathfrak g}_{\bar 0}}(x)\bigr)^n=0$ for any $x \in {\mathfrak g}_{\bar 0}$, which proves that ${\mathfrak g}_{\bar 0}$ is a nilpotent Lie algebra.
\end{proof}

\begin{rem}\label{lemago2}
The previous proposition also holds in case ${\mathfrak g}_{\bar 1}$ is a filiform ${\mathfrak g}_{\bar 0}$-module. That is, if $({\mathfrak g}={\mathfrak g}_{\bar 0}\oplus {\mathfrak g}_{\bar 1},B)$ is an odd-quadratic Lie superalgebra such that the ${\mathfrak g}_{\bar 0}$-module ${\mathfrak g}_{\bar 1}$ has the structure of filiform ${\mathfrak g}_{\bar 0}$-module then the Lie algebra ${\mathfrak g}_{\bar 0}$ is nilpotent.
\end{rem}

\begin{prop}
Let $({\mathfrak g}={\mathfrak g}_{\bar 0}\oplus {\mathfrak g}_{\bar 1},B)$ be an odd-quadratic Lie superalgebra such that the ${\mathfrak g}_{\bar 0}$-module ${\mathfrak g}_{\bar 1}$ has the structure of weak filiform ${\mathfrak g}_{\bar 0}$-module. Then the Lie superalgebra $ {\mathfrak g} $ is nilpotent.
\end{prop}

\begin{proof} 
We have the result by the paper \cite{solvableLieSuperalgebra}, since by Proposition \ref{lemgo}  ${\mathfrak g}_{\bar 0}$ is a nilpotent Lie algebra, and  $\pi(x):=(\mbox{\rm ad}_{\mathfrak g}(x))|_{{\mathfrak g}_{\bar 1}}$ is nilpotent for all $x \in {\mathfrak g}_{\bar 0}$.
\end{proof} 

The center of an odd-quadratic Lie superalgebras $({\mathfrak g}={\mathfrak g}_{\bar 0}\oplus {\mathfrak g}_{\bar 1},B)$ such that the ${\mathfrak g}_{\bar 0}$-module ${\mathfrak g}_{\bar 1}$ has the structure of weak filiform ${\mathfrak g}_{\bar 0}$-module plays an important role in the description of this class of algebras. Recall that for quadratic Lie superalgebras with an invariant scalar product the center is  $\displaystyle {\mathfrak z}({\mathfrak g}) = {\mathfrak z}({\mathfrak g})_{\bar{0}} \oplus  V_1$ (see \cite[ Lemma 2.1]{quadraticFLSA}).

\begin{prop}\label{Lema 3.3}
Let $(\mathfrak{g}={\mathfrak g}_{\bar 0}\oplus {\mathfrak g}_{\bar 1},B)$ be an odd-quadratic non-abelian Lie superalgebra with $\mathfrak{g}_{\bar 0} \neq \{0\}$ and $ \mbox{\rm dim } {\mathfrak g}_{\bar 1} = m > 0,$ and such that ${\mathfrak g}_{\bar 1}$ is a weak filiform ${\mathfrak g}_{\bar 0}$-module with respect to the flag $${\mathfrak g}_{\bar 1} = V_m \supset \dots \supset V_2 \supset V_1 = \{0\}.$$
Then the dimension of the center $\displaystyle {\mathfrak z}({\mathfrak g}) = {\mathfrak z}({\mathfrak g})_{\bar{0}} \oplus {\mathfrak z}({\mathfrak g})_{\bar{1}}$ is $1, 2$ or $3$. Concretely,
\begin{itemize}
\item[i.] $\{0\} \neq {\mathfrak z}({\mathfrak g}_{\bar 0}) = {\mathfrak z}({\mathfrak g}) \cap {\mathfrak g}_{\bar 0} = {\mathfrak z}({\mathfrak g})_{\bar{0}} $ and $\mbox{\rm dim } {\mathfrak z}({\mathfrak g}_{\bar 0})=1$ (i.e., $\mbox{\rm dim }{\mathfrak z}({\mathfrak g})_{\bar 0} =1).$

\item[ii.] ${\mathfrak z}({\mathfrak g})_{\bar{1}} =  {\mathfrak z}({\mathfrak g}) \cap {\mathfrak g}_{\bar 1} \subset V_2$ and $\mbox{\rm dim } ({\mathfrak z}({\mathfrak g}) \cap {\mathfrak g}_{\bar 1}) \leq 2$ (i.e., $\mbox{\rm dim }{\mathfrak z}({\mathfrak g})_{\bar 1} \leq 2)$.  
\end{itemize}
\end{prop}

\begin{proof}
i. Since by Proposition \ref{lemgo} ${\mathfrak g}_{\bar 0}$ is a nilpotent Lie algebra and ${\mathfrak g}_{\bar 0} \neq \{0\}$, then ${\mathfrak z}({\mathfrak g}_{\bar 0}) \neq \{0\}$. Therefore  $$B([{\mathfrak z}({\mathfrak g}_{\bar 0}),{\mathfrak g}_{\bar 1}],{\mathfrak g}_{\bar 0}) = B({\mathfrak g}_{\bar 1}, [{\mathfrak z}({\mathfrak g}_{\bar 0}),{\mathfrak g}_{\bar 0}]) = \{0\}$$ and we conclude $[{\mathfrak z}({\mathfrak g}_{\bar 0}), {\mathfrak g}_{\bar 1}]=\{0\}$ because $B$ is odd and non-degenerate. It follows that ${\mathfrak z}({\mathfrak g}_{\bar 0}) \subset {\mathfrak z}({\mathfrak g})$. Then $\{0\} \neq {\mathfrak z}({\mathfrak g}_{\bar 0}) = {\mathfrak z}({\mathfrak g}) \cap {\mathfrak g}_{\bar 0}$.

It is clear that ${\mathfrak g}_{\bar 0} \oplus [{\mathfrak g}_{\bar 0},{\mathfrak g}_{\bar 1}] \subset {\mathfrak z}({\mathfrak g}_{\bar 0})^{\perp}$ and $\mbox{\rm dim } \bigl({\mathfrak g}_{\bar 0} \oplus [{\mathfrak g}_{\bar 0}, {\mathfrak g}_{\bar 1}] \bigr) = \mbox{\rm dim } {\mathfrak g} -1$. Since ${\mathfrak z}({\mathfrak g}_{\bar 0}) \neq \{0\}$, then ${\mathfrak g}_{\bar 0} \oplus [{\mathfrak g}_{\bar 0},{\mathfrak g}_{\bar 1}] = {\mathfrak z}({\mathfrak g}_{\bar 0})^{\perp}$. Consequently, $\mbox{\rm dim }{\mathfrak z}({\mathfrak g}_{\bar 0}) = 1$.

ii. From the definition of weak filiform module we get $[{\mathfrak g}_{\bar 0},V_i]= V_{i-1}$ with $\mbox{\rm dim }V_i=i$, for $i\in \{2,\dots,m\}$. We set $V_2 := \mathbb{K}u_2 \oplus \mathbb{K}v_2$ and $V_i/V_{i-1}:=\mathbb{K}e_i$ for $i \in \{3,\dots,m\}$.  Therefore, $$[X,u_2] = [X,v_2]= 0, \mbox{ for all } X \in {\mathfrak g}_{\bar 0}.$$

Taking into account this expression together with $[{\mathfrak g}_{\bar 0},V_3]= V_2$ being $V_3 = \mbox{\rm span}_{\mathbb K}\{u_2,v_2,e_3\}$ allows us to assert, without losing generality, that there exist $X_3, X'_3 \in {\mathfrak g}_{\bar 0}$ such that $$[X_3,e_3] = \eta_{32} u_2 \mbox{ with } \eta_{32} \neq 0$$ 
and $$ [X'_3,e_3] = \lambda_{32}v_2 \mbox{ with } \lambda_{32} \neq 0.$$

Now, on account of above together with $[{\mathfrak g}_{\bar 0},V_4]= V_3$ being $V_4 = \mbox{\rm span}_{\mathbb K}\{u_2,v_2,e_3,e_4\}$ we conclude that there exists $X_4 \in {\mathfrak g}_{\bar 0}$ such that $$[X_4,e_4] = \alpha_4 e_3 + \eta_{42}u_2 + \lambda_{42}v_2, \mbox{ for some } \alpha_4 \neq 0, \ \eta_{42},\lambda_{42} \in \mathbb{K}.$$
Similarly, for all $i \in \{5,\dots,m\}$ we obtain that there exists $X_i \in {\mathfrak g}_{\bar 0}$, such that $$[X_i,e_i] = \alpha_i e_{i-1} + \sum_{j=3}^{i-2} \beta_{ij}e_j + \eta_{i2}u_2 + \lambda_{i2}v_2,  \mbox{ for some } \alpha_i \neq 0 \mbox{ and } \ \beta_{ij}, \eta_{i2}, \lambda_{i2} \in \mathbb{K}.$$

Consequently, we obtain   $[{\mathfrak g}_{\bar 0},e_i]\neq\{0\}$ for $i \in \{3,\dots,m\}$. Therefore, ${\mathfrak z}({\mathfrak g}) \cap {\mathfrak g}_{\bar 1} \subset V_2$, which concludes the proof.
\end{proof}

\begin{rem}    
By the proof of the previous proposition we have ${\mathfrak z}({\mathfrak g}_{\bar 0})^{\perp} = {\mathfrak g}_{\bar 0} \oplus [{\mathfrak g}_{\bar 0},{\mathfrak g}_{\bar 1}] $ with $  [{\mathfrak g}_{\bar 0},{\mathfrak g}_{\bar 1}] = V_{m-1}$. If we set ${\mathfrak z}({\mathfrak g}_{\bar 0}) = \mathbb{K}e$, we have $B(e,u_2) = B(e,v_2) = 0$, $B(e,e_i) = 0$ for $i \in \{3,\dots,m-1\}$, and we can assume  without loss of generality  that $B(e,e_m) = 1$. \end{rem}

\begin{rem}
We note that Proposition \ref{Lema 3.3}-ii can be established for any solvable Lie superalgebra ${\mathfrak g} = {\mathfrak g}_{\bar 0} \oplus {\mathfrak g}_{\bar 1}$ with ${\mathfrak g}_{\bar 1}$ a weak filiform ${\mathfrak g}_{\bar 0}$-module, even with no odd-quadratic structure, that is,  if ${\mathfrak z}({\mathfrak g})_{\bar{1}} =  {\mathfrak z}({\mathfrak g}) \cap {\mathfrak g}_{\bar 1} \neq \{0\}$ then ${\mathfrak z}({\mathfrak g})_{\bar{1}} ={\mathfrak z}({\mathfrak g}) \cap {\mathfrak g}_{\bar 1} \subset V_2$.
\end{rem}

\begin{lem}\label{lem2}
Let ${\mathfrak g} = {\mathfrak g}_{\bar 0}\oplus {\mathfrak g}_{\bar 1}$ be a solvable Lie superalgebra with $\mbox{\rm dim } {\mathfrak g}_{\bar 1} = m > 0$ such as ${\mathfrak g}_{\bar 1}$ is a weak filiform ${\mathfrak g}_{\bar 0}$-module with respect to the flag $${\mathfrak g}_{\bar 1} = V_m \supset \dots \supset V_2 \supset V_1 = \{0\},$$ and ${\mathfrak g}$ is equipped with an odd-invariant scalar product $B : {\mathfrak g} \times {\mathfrak g} \to \KK$ on ${\mathfrak g}$. Then $u_2, v_2 \notin ({\mathfrak g}_{\bar 0} \setminus [{\mathfrak g}_{\bar 0},{\mathfrak g}_{\bar 0}])^{\perp}$.
\end{lem}

\begin{proof}
Since $B$ is invariant and $[{\mathfrak g}_{\bar 0},u_2]=\{0\}$ we have $B({\mathfrak g}_{\bar 0}, [{\mathfrak g}_{\bar 0}, u_2])=B([{\mathfrak g}_{\bar 0}, {\mathfrak g}_{\bar 0}], u_2)=0$. Now, as $B$ is odd and non-degenerate there exists $X \in {\mathfrak g}_{\bar 0} \setminus [{\mathfrak g}_{\bar 0},{\mathfrak g}_{\bar 0}]$ such that  $B(X,u_2)\neq 0$. Analogous result  can be obtained for $v_2$, which concludes the proof.
\end{proof}

\begin{thm}\label{TEOREMA__3.2}
Let $\displaystyle ({\mathfrak g}={\mathfrak g}_{\bar 0} \oplus {\mathfrak g}_{\bar 1},B)$ be an odd-quadratic Lie superalgebra with $\mbox{\rm dim }{\mathfrak g}_{\bar 1} = m > 0$, such that ${\mathfrak g}_{\bar 1}$ is a weak filiform ${\mathfrak g}_{\bar 0}$-module with respect to the flag $${\mathfrak g}_{\bar 1} = V_m \supset \dots \supset V_2
\supset V_1 = \{0\}.$$ Consider $D$ an odd skew-supersymmetric superderivation of $({\mathfrak g},B)$, $X_0 \in {\mathfrak g}_{\bar 0}$, and $\lambda_0 \in \mathbb{K}$ such that
$$ D(X_0)=0, \hspace{0.8cm} D^2=\frac{1}{2}[X_0, \cdot ]_{\mathfrak g}, \hspace{0.8cm}  e_m \in D({\mathfrak g}_{\bar 0}). $$

\noindent Define a map $\Omega : (\mathbb{K}e)_{\bar 1} \to Der({\mathfrak g} \oplus P(\mathbb{K}e^{*}))$ by $\Omega (e) := \widetilde{D}$, where $\widetilde{D}: {\mathfrak g} \oplus P(\mathbb{K}e^{*}) \to {\mathfrak g} \oplus P(\mathbb{K}e^{*})$ satisfies $\widetilde{D}(e^{*})=0$ and 
$$\widetilde{D}(X) = D(X) - (-1)^{\bar i} B(X,X_0)e^*, \hspace{0.1cm} \mbox{\rm for all } X \in {\mathfrak g}_{\bar i}.$$

\noindent Consider the bilinear map $\zeta : \mathbb{K}e \times \mathbb{K}e \to {\mathfrak g} \oplus P(\mathbb{K}e^{*})$ defined by $\zeta (e,e):=X_0 + \lambda_0e^{*}$. Then 
${\mathfrak t} = \mathbb{K}e \oplus {\mathfrak g} \oplus P(\mathbb{K}e^{*})$ endowed with the even skew-symmetric bilinear map $[\cdot,\cdot]: {\mathfrak t} \times {\mathfrak t} \to {\mathfrak t}$ defined by 
\begin{align*}
& [e,e] = X_0+\lambda_0 e^{*}, \\
& [e,X] = D(X)-(-1)^{\bar i} B(X,X_0) e^*, \hspace{0.1cm} \mbox{\rm for all } X \in {\mathfrak g}_{\bar i},\\
& [X,Y] = [X,Y]_{\mathfrak g}+B(D(X),Y)e^*, \hspace{0.1cm} \mbox{\rm for all } X,Y \in {\mathfrak g},\\
& [e^*, {\mathfrak t}] = 0,
\end{align*}
\noindent is a Lie superalgebra. More precisely, $(\mathfrak{t},[\cdot,\cdot])$ is the generalized semi-direct product of ${\mathfrak g} \oplus P(\mathbb{K}e^{*})$ by the $1$-dimensional Lie superalgebra $(\mathbb{K}e)_{\bar 1}$ (by means of $\Omega$ and $\zeta$). Moreover, the supersymmetric bilinear form 
$\widetilde{B}: {\mathfrak t} \times {\mathfrak t} \to \mathbb{K}$ defined by 
\begin{align*}
& \widetilde{B}|_{\mathfrak g \times \mathfrak g}=B,\\
&\widetilde{B}(e,e^*)=1,\\
& \widetilde{B}(\mathfrak g, e)=\widetilde{B}(\mathfrak g,e^*)=\{0\},
\end{align*}
\noindent is an odd-invariant scalar product on ${\mathfrak t}$. In this case, we say that $({\mathfrak t}={\mathfrak t}_{\bar 0}\oplus {\mathfrak t}_{\bar 1}, \widetilde{B})$ is the generalized odd double extension of $(\mathfrak g, B)$, by the $1$-dimensional Lie superalgebra $(\mathbb{K}e)_{\bar 1}$ (by means of $D$, $X_0$, and $\lambda_0$). Furthermore, ${\mathfrak t}_{\bar 1}$ is a weak filiform ${\mathfrak t}_{\bar 0}$-module with respect to the flag $${\mathfrak t}_{\bar 1}=\widetilde V_{m+1} \supset \dots \supset \widetilde V_2
\supset \widetilde V_1 = \{0\}$$ being $\widetilde V_{m+1}:= \mathbb{K}e \oplus V_m$, $\widetilde V_i:=V_i,$ with $i \in \{1,\dots,m\}$. 
\end{thm}

\begin{proof}
Taking into account \cite[Theorem 2.14]{OddQuadraticLieSuperalgebras}, only rest to check the structure of weak filiform 
${\mathfrak t}_{\bar 0}$-module. Since ${\mathfrak g}_{\bar 1}$ is a weak filiform ${\mathfrak g}_{\bar 0}$-module with respect to the flag $${\mathfrak g}_{\bar 1}=V_m \supset \dots \supset V_2 \supset V_1=\{0\},$$ we have that  $[{\mathfrak g}_{\bar 0}, V_{j}]_{\mathfrak g}=V_{j-1}$ for all $j \in \{2,\dots,m\}$. On account of both $D$ and $B$ are odd we get  $B(D({\mathfrak g}_{\bar 0}),V_j))=0$ for $j \in \{2,\dots,m\}$. This fact together with 
${\mathfrak t}_{\bar 0}={\mathfrak g}_{\bar 0} \oplus P(\mathbb{K}e^{*})$, being $e^{*}$ a central element for ${\mathfrak t}$, that is, $[e^{*}, {\mathfrak t}]=0$,  allow us to assert 
$$[{\mathfrak t}_{\bar 0}, V_{j}]=[{\mathfrak g}_{\bar 0}, V_{j}]_{\mathfrak g}=V_{j-1}, \quad 2 \leq j \leq m.$$

Now, since $\widetilde V_i:=V_i$ for $i \in \{1,\dots,m\}$ we get  
$$[{\mathfrak t}_{\bar 0}, \widetilde V_{j}]=\widetilde V_{j-1}, \quad 1 \leq j \leq m.$$
Only rest to check $[{\mathfrak t}_{\bar 0}, \widetilde V_{m+1}]$. Note that $$[{\mathfrak t}_{\bar 0}, \widetilde V_{m+1}] = [{\mathfrak g}_{\bar 0}, \mathbb{K}e \oplus V_m]$$
with  $[{\mathfrak g}_{\bar 0}, V_{m}]=V_{m-1}$ and $[{\mathfrak g}_{\bar 0}, e]=-D({\mathfrak g}_{\bar 0})$. As $e_m \in D({\mathfrak g}_{\bar 0})$ then $$[{\mathfrak t}_{\bar 0}, \widetilde V_{m+1}] = V_{m-1} \oplus \mathbb{K}e_m = V_m = \widetilde V_m$$ which concludes the proof of the theorem.
\end{proof}

\section{Inductive description of odd-quadratic solvable Lie superalgebras of weak filiform type}

\begin{lem}\label{lem3}
Let $({\mathfrak g} = {\mathfrak g}_{\bar 0} \oplus {\mathfrak g}_{\bar 1},B)$ be an odd-quadratic  Lie superalgebra with $\mbox{\rm dim } {\mathfrak g}_{\bar 1} = m > 0$ such as ${\mathfrak g}_{\bar 1}$ has the structure of  weak filiform ${\mathfrak g}_{\bar 0}$-module with respect to the flag $${\mathfrak g}_{\bar 1}=V_m \supset \dots \supset V_2
\supset V_1=\{ 0 \}.$$ Then $e_m \notin \bigl({\mathfrak z}({\mathfrak g}\bigr)_{\bar 0})^{\perp}$.
\end{lem}

\begin{proof}
Since ${\mathfrak g}_{\bar 1}$ has the structure of a weak filiform ${\mathfrak g}_{\bar 0}$-module with respect to the flag $${\mathfrak g}_{\bar 1}=V_m \supset \dots \supset V_2
\supset V_1=\{ 0 \},$$ being $ V_i/V_{i-1}:=\mathbb{K}e_i$, $i \in \{3,\dots,m\},$ and $V_2=span_{\mathbb{K}}\{ u_2,v_2\}$, we can suppose that for each $i \in \{4,\dots,m\}$ there exits $t_i \in {\mathfrak g}_{\bar 0}$ such that $$[t_i,e_i] = e_{i-1}.$$ Also we have $t_3$ and $t_2$ verifying $[t_3,e_3]=u_2$ and $[t_2,e_3]=v_2$. In general, such $t_i$ is a linear combination of the homogeneous basis selected for ${\mathfrak g}_{\bar 0}$ and it is not unique. Now, let $e^*$ be a non-zero element of ${\mathfrak z}({\mathfrak g})_{\bar 0},$ so $[e^*,t_i]=0$ and as $B$ is invariant we get $$0 = B([e^*,t_i],e_i)=B(e^*,[t_i,e_i])=B(e^*,e_{i-1}), \mbox{\rm  for } i \in \{4,\dots,m\}.$$ Moreover,  
$$0 = B([e^*,t_3],e_3)=B(e^*,[t_3,e_3])=B(e^*,u_2)$$
and 
$$0 = B([e^*,t_2],e_3)=B(e^*,[t_2,e_3])=B(e^*,v_2).$$
Finally, taking into account that $B$ is non-degenerate it follows that $B(e^*,e_m)\neq 0$ which concludes the proof of the lemma.
\end{proof}

\noindent Since for  any odd-quadratic Lie superalgebra $({\mathfrak g} = {\mathfrak g}_{\bar 0} \oplus {\mathfrak g}_{\bar 1}, B)$ with $\mbox{\rm dim }{\mathfrak g}_{\bar 1} = m > 0$  such that ${\mathfrak g}_{\bar 1}$ is a weak filiform ${\mathfrak g}_{\bar 0}$-module the even part of the center is non-zero, this is    ${\mathfrak z} ({\mathfrak g})_{\bar 0} \neq \{0\}$ (see Proposition \ref{Lema 3.3}i.), we can prove the converse of Theorem \ref{TEOREMA__3.2}.

\begin{thm} \label{generalized}
Let $({\mathfrak g} = {\mathfrak g}_{\bar 0} \oplus {\mathfrak g}_{\bar 1}, B)$ be an odd-quadratic Lie superalgebra with $\mbox{\rm dim }{\mathfrak g}_{\bar 1} = m > 0$ and $\mbox{\rm dim }{\mathfrak g} > 1$, such that ${\mathfrak g}_{\bar 1}$ is a weak filiform ${\mathfrak g}_{\bar 0}$-module with respect to the flag $${\mathfrak g}_{\bar 1} = V_m \supset \dots \supset V_2 \supset V_1=\{0\}.$$  Then $({\mathfrak g},B)$ is a generalized odd double extension of an odd-quadratic Lie superalgebra $({\mathfrak h}, \widetilde{B})$ (such that $\mbox{\rm dim }{\mathfrak h} = \mbox{\rm dim }{\mathfrak g}-2$) by the $1$-dimensional Lie superalgebra  $(\mathbb{K}e_m)_{\bar 1}$  being $V_m/V_{m-1} := \mathbb{K}e_m$. Furthermore, ${\mathfrak h}_{\bar 1}$ is a weak \textit{filiform ${\mathfrak h}_{\bar 0}$-module} with respect to the flag $${\mathfrak h}_{\bar 1} = V_{m-1} \supset \dots \supset V_2 \supset V_1=\{0\}.$$
\end{thm}

\begin{proof}
Let $e^*$ be a non-zero element of $\mathfrak z(\mathfrak g) _{\bar{0}}$ and denote $I:=\mathbb{K} e^{*}$. As $B$ is odd we have ${\mathfrak g}_{\bar 0} \subset I^{\perp}$. By using Lemma \ref{lem3} we get $B(e^*,e_m) \neq 0$ and $B(e^*,u_2)=B(e^*,v_2)=0$ and $B(e^*,e_i)=0$ for $i \in \{3,\dots,m-1\}$. Therefore ${\mathfrak g}=I^{\perp} \oplus \mathbb{K} e_m$ and we may assume that $B(e^*,e_m)=1$. Considering $A:=\mathbb{K} e^{*} \oplus \mathbb{K} e_m$, ${\mathfrak h} := A^{\perp}$ with respect to $B$ and $\widetilde{B}=B|_{{\mathfrak h} \times {\mathfrak h}},$ and following the proof of \cite[Proposition 2.15]{OddQuadraticLieSuperalgebras}, we have
\begin{align*}
&[X,Y] = \alpha(X,Y) + \varphi(X,Y)e^*, \mbox{\rm for }X,Y \in {\mathfrak h}, \mbox{\rm with } \alpha(X,Y) \in {\mathfrak h}, \varphi(X,Y) \in \mathbb{K}.\\
&[e_m,X]=D(X)+\psi(X)e^*, \mbox{\rm for any } X \in {\mathfrak h}, \mbox{\rm where } D(X) \in {\mathfrak h}, \psi(X) \in \mathbb{K}.
\end{align*}

\noindent As $e_m \in {\mathfrak g}_{\bar 1}$ then $[e_m,e_m]$ is not necessarily zero and we can write
\begin{align*}
&[e_m,e_m]=X_0+\lambda_0 e^*, \mbox{\rm being } X_0 \in {\mathfrak g}_{\bar 0}, \lambda_0 \in \mathbb{K}.  \end{align*}

\noindent {\bf Claim 1.} The triple $({\mathfrak h}, \alpha=[\cdot,\cdot]_{\mathfrak h},\widetilde{B})$ is an odd-quadratic Lie superalgebra with ${\mathfrak h}_{\bar 1}$ a weak filiform ${\mathfrak h}_{\bar 0}$-module with respect to the flag ${\mathfrak h}_{\bar 1} = V_{m-1} \supset \dots \supset V_2 \supset V_1=\{ 0\}.$

\

\begin{dem}
The fact that $({\mathfrak h}, \alpha=[\cdot,\cdot]_{\mathfrak h},\widetilde{B})$ is an odd-quadratic Lie superalgebra follows from the proof of \cite[Proposition 2.15]{OddQuadraticLieSuperalgebras}. The structure of  weak filiform ${\mathfrak h}_{\bar 0}$-module on ${\mathfrak h}_{\bar 1}$ is naturally inherited from the structure of weak filiform ${\mathfrak g}_{\bar 0}$-module on ${\mathfrak g}_{\bar 1}$. 
\end{dem}

\

\noindent {\bf Claim 2.} $D$ is an odd skew-supersymmetric superderivation of $({\mathfrak h}, \widetilde{B})$ verifying 
$$ D(X_0)=0, \hspace{0.8cm} D^2=\frac{1}{2}[X_0, \cdot ]_{\mathfrak h}, \hspace{0.8cm}  e_{m-1} \in D({\mathfrak g}_{\bar 0}), $$ being $ V_{m-1}/V_{m-2}:=\mathbb{K}e_{m-1}$.

\

\begin{dem2}
On account of  \cite{OddQuadraticLieSuperalgebras} only remains to check that $e_{m-1} \in D({\mathfrak g}_{\bar 0})$  being $V_{m-1}/V_{m-2}:=\mathbb{K}e_{m-1}$. Since ${\mathfrak g}_{\bar 1}$ has the structure of a weak filiform ${\mathfrak g}_{\bar 0}$-module with respect to the flag $${\mathfrak g}_{\bar 1}=V_m \supset \dots \supset V_2
\supset V_1=\{0\},$$ being $ V_i/V_{i-1}:=\mathbb{K}e_i$ for $i \in \{3,\dots,m\}$ and $V_2=span_{\mathbb{K}}\{u_2,v_2\}$, we can suppose  that for each $i \in \{4,\dots,m\}$ there exist $t_i \in {\mathfrak g}_{\bar 0}$ such that $[ t_i,e_i]=e_{i-1}$. In particular, we get $t_m$ such that $[t_m, e_m] = e_{m-1}$. This fact together with $[e_m,t_m]=D(t_m)$ leads to $e_{m-1} = D(t_m) \in D({\mathfrak g}_{\bar 0})$. 
\end{dem2}
\end{proof}
\
 
\noindent The following result gives us the main result of the paper, that is, a full description of all odd-quadratic Lie superalgebras of weak filiform type using generalizations of double extensions and/or orthogonal direct sums.

\begin{cor}
An odd-quadratic Lie superalgebra of weak filiform type is obtained from a $6$-dimensional weak filiform superalgebra $\{ X_1,X_2,X_3,e_3,v_2,u_2 \}$ by a sequence of   generalized odd double extensions by  $1$-dimensional Lie superalgebras and/or direct sums of them.
\end{cor}

To complete the study we classify all the complex odd-quadratic Lie superalgebras of weak filiform type for the smallest possible dimensions which makes sense, that is, $\mbox{{\rm dim} }{\mathfrak g}_{\bar 0} = \mbox{{\rm dim} }{\mathfrak g}_{\bar 1}=3$ and $\mbox{{\rm dim} }{\mathfrak g}_{\bar 0}= \mbox{{\rm dim} }{\mathfrak g}_{\bar 1}=4$. Then, we have the following result.

\begin{thm}
If $\displaystyle ({\mathfrak g} = {\mathfrak g}_{\bar 0} \oplus {\mathfrak g}_{\bar 1}, B)$ is a $6$-dimensional complex odd-quadratic Lie superalgebra of weak filiform type, then ${\mathfrak g}$ is isomorphic either ${\mathfrak g}^0_6$ or ${\mathfrak g}^1_6$, being ${\mathfrak g}^{\delta} \  (\delta \in \{ 0,1\})$ the family of Lie superalgebras which can be expressed in a suitable basis $\{X_1,X_2,X_3,e_3,u_2,v_2\}$ by the following non-null bracket products

$${\mathfrak g}^{\delta}_6 :
\left\{ \begin{array}{ll}
    [X_1,X_2]=X_3, &  \\{}
    [X_1,e_3]=u_2, & \\{}
    [X_2,e_3]=v_2, & \\{}
    [e_3,e_3]=\delta X_3, & \delta \in \{ 0,1\},
\end{array}\right.$$
where $\{X_1,X_2,X_3\}$ are even basis vectors and $\{e_3,u_2,v_2\}$ odd ones. All the above bracket products are skew-symmetric except $[e_3,e_3] $ which is symmetric. Moreover, for both cases, ${\mathfrak g}^0_6$ and ${\mathfrak g}^1_6$,  the only possible non-null values of $B$ are determined by:
$$\begin{array}{ll}
B(X_1,v_2)=\lambda, & B(X_1,e_3)=\alpha, \\
B(X_2,u_2)=-\lambda, & B(X_2,e_3)=\beta, \\
B(X_3,e_3)=\lambda, &  \\
\end{array}$$
with $\lambda \neq 0$  and $\alpha, \ \beta \in \mathbb{C}$.
\end{thm}

\begin{proof}
 Let $\displaystyle ({\mathfrak g} = {\mathfrak g}_{\bar 0} \oplus {\mathfrak g}_{\bar 1}, B)$ be a complex odd-quadratic Lie superalgebra with ${\rm dim}({\mathfrak g}_{\bar 0})= {\rm dim} ({\mathfrak g}_{\bar 1})=3$, such that ${\mathfrak g}_{\bar 1}$ is a weak filiform ${\mathfrak g}_{\bar 0}$-module with respect to the flag $${\mathfrak g}_{\bar 1} = V_3 \supset  V_2 \supset V_1=\{0\}.$$
 
We set  ${V}_3/ {V}_{2} := \mathbb{C}e_3$ and   ${V}_2 =span_{ \mathbb{C}} \{u_2,v_2 \}$. From Proposition \ref{lemgo}, on the other hand, we have that ${\mathfrak g}_{\bar 0}$ is nilpotent  and Proposition \ref{Lema 3.3}  leads to $\mbox{\rm dim }{\mathfrak z}({\mathfrak g}_{\bar 0})=1$. Since the only non-abelian $3$-dimensional nilpotent lie algebra is defined by $[X_1,X_2]=X_3$ we take this as our ${\mathfrak g}_{\bar 0}$. Moreover, $X_3$ is not only a central element for ${\mathfrak g}_{\bar 0}$ but for the whole Lie superalgebra ${\mathfrak g} = {\mathfrak g}_{\bar 0} \oplus {\mathfrak g}_{\bar 1}$ (see Proposition \ref{Lema 3.3}). This latter fact implies, in particular, that $[X_3,V_3]=\{ 0\}$. 

On account of ${\mathfrak g}_{\bar 1}$ is a weak filiform ${\mathfrak g}_{\bar 0}$-module we have that $[{\mathfrak g}_{\bar 0}, V_3]=V_2$ and $[{\mathfrak g}_{\bar 0}, V_2]=\{ 0\}$. Therefore, there is no loss of generality in supposing $$ [X_1,e_3]=u_2 \qquad  {\rm and} \qquad
    [X_2,e_3]=v_2. $$ 
    
Now, for having totally described the multiplication table of the Lie superalgebra ${\mathfrak g} = {\mathfrak g}_{\bar 0} \oplus {\mathfrak g}_{\bar 1} $ only rest to determine the symmetric  bracket products  $[{\mathfrak g}_{\bar 1}, {\mathfrak g}_{\bar 1}]$, i.e.:
$$[e_3,e_3], \ [e_3,u_2], \ [e_3,v_2], \ [u_2,u_2], \ [u_2,v_2], \ [v_2,v_2].$$

Firstly, we set $[e_3,e_3]=aX_1+bX_2+cX_3$. After,  by applying the super Jacobi identity $$[x,[y,z]]= [[x,y],z]- (-1)^{|y| |z|}  [[x,z],y]$$
for the  triple $\{x,y,z\}$  we get the  following constrains given
 in the table:
 \begin{center}
 	\begin{tabular}{l|l}
 		
 		Super Jacobi Identity & \quad Constrain\\
 		\hline \hline
 			$\{e_3,e_3,e_3\}$ & \quad $a=b=0$ \\
 			$\{X_1,e_3,e_3\}$ & \quad $[u_2,e_3]=0$ \\
 			$\{X_2,e_3,e_3\}$ & \quad $[v_2,e_3]=0$ \\
 			$\{X_1,e_3,u_2\}$ & \quad $[u_2,u_2]=0$ \\
 			$\{X_2,e_3,v_2\}$ & \quad $[v_2,v_2]=0$ \\
 			$\{X_1,e_3,v_2\}$ & \quad $[u_2,v_2]=0$ \\
 			\end{tabular}
 \end{center}
remaining only $[e_3,e_3]=cX_3$. Let us remark that the cases $c=0$ and $c \neq 0$ are clearly non-isomorphic. In fact, the first case give us a  Lie algebra (more concretely a $\mathbb{Z}_2$-graded Lie algebra) and the second case leads,  for any $c \neq 0$, to a Lie superalgebra which is not a Lie algebra because it contains non-null symmetric bracket products. In this latter case, the isomorphism (change of scale) defined by $$X'_i=X_i, \ 1 \leq i \leq 3, \ e'_3=\frac{1}{\sqrt{c}}e_3, \ u'_2=\frac{1}{\sqrt{c}}u_2, \ v'_2=\frac{1}{\sqrt{c}}v_2$$
allows to assert that $c$ can be supposed $c=1$, obtaining then ${\mathfrak g}^0_6$ and ${\mathfrak g}^1_6$ of the statement of the Theorem. 

Finally, we study the bilinear form $B$. Since $B$ is odd and supersymmetrical ($B(X,Y)=B(Y,X)$) it will be totally determined by the following values:
$$\begin{array}{lll}
B(X_1,e_3), & B(X_1,u_2),& B(X_1,v_2), \\
B(X_2,e_3), & B(X_2,u_2),& B(X_2,v_2), \\
B(X_3,e_3), & B(X_3,u_2),& B(X_3,v_2). \\
\end{array}$$

By applying the invariant condition $B([x,y],z)=B(x,[y,z])$
for the  ordered triple $\{x,y,z\}$  we get the  following relationships given
 in the table:
 \begin{center}
 	\begin{tabular}{l|l}
 		
 		Invariant Condition & \quad Relationship\\
 		\hline \hline
 			$\{X_1,X_1,e_3\}$ & \quad $B(X_1,u_2)=0$ \\
 			$\{X_2,X_2,e_3\}$ & \quad $B(X_2,v_2)=0$ \\
 			$\{X_3,X_1,e_3\}$ & \quad $B(X_3,u_2)=0$ \\
 			$\{X_3,X_2,e_3\}$ & \quad $B(X_3,v_2)=0$ \\
 			$\{X_1,X_2,e_3\}$ & \quad $B(X_3,e_3)=B(X_1,v_2)$ \\
 			$\{X_2,X_1,e_3\}$ & \quad $B(X_3,e_3)=-B(X_2,u_2)$ \\
 			\end{tabular}
 \end{center}
Since $B$ is non-degenerate we get $B(X_3,e_3)\neq 0$ and after renaming  $B(X_3,e_3)= \lambda$, $B(x_1,e_3)=\alpha$ and $B(x_2,e_3)=\beta$ we obtain the expression of the statement of the Theorem, which concludes the proof.
\end{proof}

Next, and previous to  study the case ${\rm dim}({\mathfrak g}_{\bar 0})= {\rm dim} ({\mathfrak g}_{\bar 1})=4$ we show a lemma which will be useful for the next classification Theorem.

\begin{lem}\label{descending}
Let ${\mathfrak g} = {\mathfrak g}_{\bar 0} \oplus {\mathfrak g}_{\bar 1}$ be a Lie superalgebra such that ${\mathfrak g}_{\bar 1}$ is a weak filiform ${\mathfrak g}_{\bar 0}$-module with respect to the flag $${\mathfrak g}_{\bar 1} = V_m \supset \dots \supset V_2 \supset V_1=\{0\}.$$
$[{\mathfrak g}_{\bar 0},V_i]=V_{i-1}$ $2 \leq i \leq m$. If we denote by $\mathcal{C}^{i}{\mathfrak g}_{\bar 0}$ the descending central sequence of the Lie algebra ${\mathfrak g}_{\bar 0}$, i.e. $\mathcal{C}^{0}{\mathfrak g}_{\bar 0}:={\mathfrak g}_{\bar 0}$ and $\mathcal{C}^{i}{\mathfrak g}_{\bar 0}:=[\mathcal{C}^{i-1}{\mathfrak g}_{\bar 0}, {\mathfrak g}_{\bar 0}]$ with $i \geq 1$, then it is verified that $$[\mathcal{C}^{j}{\mathfrak g}_{\bar 0},V_i]\subset V_{i-j-1}, \quad \mbox{ for all } i\leq m, \ j \geq 1$$
being $V_n:=\{0\}$ for all $n \leq 1$.
\end{lem}
\begin{proof}
Let us start with the first case $j=1$. Any element of $\mathcal{C}^{1}{\mathfrak g}_{\bar 0}$ can be expressed as a linear combination of bracket products $[X,Z]$ with $X, \ Z \in {\mathfrak g}_{\bar 0}$ and by super Jacobi identity we have that $[\mathcal{C}^{1}{\mathfrak g}_{\bar 0},V_i]\subset V_{i-2}$. More precisely

$$[V_i,[X,Z]]= [[V_i,X],Z]-  [[V_i,Z],X]\subset [[V_i,{\mathfrak g}_{\bar 0}],{\mathfrak g}_{\bar 0}]=[V_{i-1},{\mathfrak g}_{\bar 0}]=V_{i-2}.$$
Now by induction, supposing the result holds for $j$, for $j+1$ we get  
$$\begin{array}{ll}[V_i,\mathcal{C}^{j+1}{\mathfrak g}_{\bar 0}]= &[V_i,[\mathcal{C}^{j}{\mathfrak g}_{\bar 0},{\mathfrak g}_{\bar 0} ]]= [[V_i,\mathcal{C}^{j}{\mathfrak g}_{\bar 0}],{\mathfrak g}_{\bar 0} ]-  [[V_i,{\mathfrak g}_{\bar 0}],\mathcal{C}^{j}{\mathfrak g}_{\bar 0}]\subset \\ \\
& \subset [V_{i-j-1},{\mathfrak g}_{\bar 0} ]-  [V_{i-1},\mathcal{C}^{j}{\mathfrak g}_{\bar 0}] \subset V_{i-j-2}.\end{array}$$
\end{proof}
\begin{thm}
If $\displaystyle ({\mathfrak g} = {\mathfrak g}_{\bar 0} \oplus {\mathfrak g}_{\bar 1}, B)$ is an $8$-dimensional complex odd-quadratic Lie superalgebra of weak filiform type, then ${\mathfrak g}$ is isomorphic to one of the following pairwise non-isomorphic Lie superalgebras: ${\mathfrak g}^0_8$, ${\mathfrak g}^1_8$ or ${\mathfrak g}^2_8$.  These Lie superalgebras  can be expressed in a suitable basis $\{X_1,X_2,X_3,X_4,e_4,e_3,u_2,v_2\}$ by the following non-null bracket products

$$\begin{array}{ll}{\mathfrak g}^{0}_8 :
\left\{ \begin{array}{ll}
    [X_1,X_2]=X_3, &  \\{}
    [X_1,X_3]=X_4, &  \\{}
    [X_1,e_3]=u_2, & \\{}
    [X_2,e_3]=v_2, & \\{}
    [X_1,e_4]=e_3, &  \\{}
    [X_3,e_4]=-v_2. &  
\end{array}\right. & {\mathfrak g}^{1}_8 :
\left\{ \begin{array}{ll}
    [X_1,X_2]=X_3, &  [e_4,e_4]=X_4,   \\{}
    [X_1,X_3]=X_4, &  \\{}
    [X_1,e_3]=u_2, & \\{}
    [X_2,e_3]=v_2, & \\{}
    [X_1,e_4]=e_3, &  \\{}
    [X_3,e_4]=-v_2.
     \end{array}\right.\end{array}$$
     
$${\mathfrak g}^{2}_8 :
\left\{ \begin{array}{ll}
    [X_1,X_2]=X_3, &  [e_4,e_4]=X_2\\{}
    [X_1,X_3]=X_4, &  [e_3,e_4]=\frac{1}{2}X_3\\{}
    [X_1,e_3]=u_2, & [e_3,e_3]=-\frac{1}{2}X_4\\{}
    [X_2,e_3]=v_2, & [u_2,e_4]=X_4\\{}
    [X_1,e_4]=e_3, &  \\{}
    [X_3,e_4]=-v_2, &  
\end{array}\right.$$
where $\{X_1,X_2,X_3,X_4\}$ are even basis vectors and $\{e_4,e_3,u_2,v_2\}$ odd ones. All the above bracket products are skew-symmetric except those involving two odd vectors which are symmetric. 
Moreover, for the three cases, ${\mathfrak g}^0_8$, ${\mathfrak g}^1_8$ and ${\mathfrak g}^2_8$,  the only possible non-null values of $B$ are determined by:
$$\begin{array}{ll}
B(X_1,v_2)=\lambda, & B(X_1,e_4)=\alpha, \\
B(X_2,u_2)=-\lambda, & B(X_2,e_4)=\beta, \\
B(X_3,e_3)=\lambda, &  \\
B(X_4,e_4)=-\lambda, &  \\
\end{array}$$
with $\lambda \neq 0$  and $\alpha, \ \beta \in \mathbb{C}$.
\end{thm}

\begin{proof}
Firstly, let us note that the three Lie superalgebras of the statement are clearly pairwise non-isomorphic. In particular, we have that ${\rm dim}({\mathfrak z} ({\mathfrak g}^2_8))=2$ whereas ${\rm dim}({\mathfrak z} ({\mathfrak g}^{i}_8))=3$, $i \in \{ 0,1\}$. Let us remark that ${\mathfrak g}^0_8$ and ${\mathfrak g}^1_8$ are clearly non-isomorphic. In fact, the first one is a  Lie algebra (more concretely a $\mathbb{Z}_2$-graded Lie algebra) and the second one is a Lie superalgebra which is not a Lie algebra because it contains non-null symmetric bracket products.

 Now, suppose $\displaystyle ({\mathfrak g} = {\mathfrak g}_{\bar 0} \oplus {\mathfrak g}_{\bar 1}, B)$  a complex odd-quadratic Lie superalgebra with ${\rm dim}({\mathfrak g}_{\bar 0})= {\rm dim} ({\mathfrak g}_{\bar 1})=4$, such that ${\mathfrak g}_{\bar 1}$ is a weak filiform ${\mathfrak g}_{\bar 0}$-module with respect to the flag $${\mathfrak g}_{\bar 1} = V_4 \supset V_3 \supset  V_2 \supset V_1=\{0\}.$$
 
We set  ${V}_4/ {V}_{3} := \mathbb{C}e_4$, ${V}_3/ {V}_{2} := \mathbb{C}e_3$ and   ${V}_2 =span_{ \mathbb{C}} \{u_2,v_2 \}$. On the other hand, from Propositions \ref{lemgo} and \ref{Lema 3.3} we have that ${\mathfrak g}_{\bar 0}$ is nilpotent  and  $\mbox{\rm dim }{\mathfrak z}({\mathfrak g}_{\bar 0})=1$. There is only one $4$-dimensional Lie algebra, up to isomorphism, verifying these two conditions (see for instance \cite{Graaf}): $[X_1,X_2]=X_3$, $[X_1,X_3]=X_4$ so we take this as  ${\mathfrak g}_{\bar 0}$. Note that $X_4$ is  a central element for the Lie superalgebra ${\mathfrak g} = {\mathfrak g}_{\bar 0} \oplus {\mathfrak g}_{\bar 1}$ (see Proposition \ref{Lema 3.3}). Thus, in particular,  $[X_4,{\mathfrak g}_{\bar 1}]=\{ 0\}$. 

Analogous to the $6$-dimensional case, from ${\mathfrak g}_{\bar 1}$ being a weak filiform ${\mathfrak g}_{\bar 0}$-module we have that $[{\mathfrak g}_{\bar 0}, V_3]=V_2$ and $[{\mathfrak g}_{\bar 0}, V_2]=\{ 0\}$. Then, there is no loss of generality in supposing $ [X_1,e_3]=u_2$ and
    $[X_2,e_3]=v_2.$ Now, since $X_3, \ X_4 \in \mathcal{C}^{1}{\mathfrak g}_{\bar 0}$ and from Lemma \ref{descending} we get $[X_3,e_3]=[X_4,e_3]=0$ . Also, Lemma \ref{descending} allows to set  
    
    $$[X_1,e_4]=c_{14}e_3+a_{14}u_2+b_{14}v_2$$
  $$[X_2,e_4]=d_{24}e_3+a_{24}u_2+b_{24}v_2$$
 with $c_{14}$ or $d_{24} \neq 0$. The super Jacobi identity for the triple $\{ e_4,X_1,X_2\}$ leads to $[X_3,e_4]=d_{24}u_2-c_{14}v_2$. Thus, the skew-symmetric bracket products of the Lie superalgebra are exactly:
 
 $$\begin{array}{ll}
    [X_1,X_2]=X_3, &  [X_1,e_4]=c_{14}e_3+a_{14}u_2+b_{14}v_2\\{}
    [X_1,X_3]=X_4, &  [X_2,e_4]=d_{24}e_3+a_{24}u_2+b_{24}v_2\\{}
    [X_1,e_3]=u_2, & [X_3,e_4]=d_{24}u_2-c_{14}v_2,\\{}
    [X_2,e_3]=v_2. & 
    \end{array}$$
Next, by means of a sequence of three isomorphisms we show that always can be supposed $c_{14}=1$ and $a_{14}=b_{14}=0$. First, we see that $c_{14}$ can be supposed $c_{14} \neq 0$. In fact, if $c_{14}=0$ and $d_{24}\neq 0$ through the isomorphism where all the basis vectors remain invariant except for $X_1$ and $u_2$ whose new values $X'_1$,  $u'_2$ are
$$  X'_1=X_1+X_2, \ u'_2=u_2+v_2$$
we obtain the new structure constant $c'_{14}=d_{24} \neq 0$. Secondly, by the change of scale (isomorphism) where all the basis vectors remain invariant except for 
$ \{ e'_4=\frac{1}{c_{14}}e_4\}$, we get $c'_{14}=1$. Finally, by renaming (isomorphism) $e'_3:=e_3+a_{14}u_2+b_{14}v_2$ we obtain 
$$\begin{array}{ll}
    [X_1,X_2]=X_3, &  [X_1,e_4]=e_3\\{}
    [X_1,X_3]=X_4, &  [X_2,e_4]=d_{24}e_3+a_{24}u_2+b_{24}v_2\\{}
    [X_1,e_3]=u_2, & [X_3,e_4]=d_{24}u_2-v_2,\\{}
    [X_2,e_3]=v_2. & 
    \end{array}$$

Only rest to determine the symmetric  bracket products  belonging to $[{\mathfrak g}_{\bar 1}, {\mathfrak g}_{\bar 1}]$ for having totally described the multiplication table of the Lie superalgebra. We start by setting $$[e_4,e_4]:=a_{44}^1X_1+a_{44}^2X_2+a_{44}^3X_3+a_{44}^4X_4$$ 
$$[e_3,e_3]:=a_{33}^1X_1+a_{33}^2X_2+a_{33}^3X_3+a_{33}^4X_4$$

After,  by applying the super Jacobi identity 
for the  triple $\{x,y,z\}$  we get the  following constrains given
 in the table:
  \begin{center}
 	\begin{tabular}{l|l}
 		
 		Super Jacobi Identity & \quad Constrain\\
 		\hline \hline 
 			$\{X_1,e_4,e_4\}$ & \quad $[e_3,e_4]=\frac{1}{2}(a_{44}^2X_3+a_{44}^3X_4)$ \\
 			$\{e_3,e_3,e_3\}$ & \quad $a_{33}^1=a_{33}^2=0$ \\
 			$\{X_1,e_3,e_4\}$ & \quad $[u_2,e_4]=-a_{33}^3X_3+(\frac{1}{2}a_{44}^2-a_{33}^4)X_4$ \\
 			$\{X_1,u_2,e_4\}$ & \quad $[u_2,e_3]=-a_{33}^3X_4$ \\
 			$\{X_1,e_3,e_3\}$ & \quad $a_{33}=0$ \\
 			$\{X_2,e_3,e_3\}$ & \quad $[v_2,e_3]=0$ \\
 			$\{X_1,u_2,e_3\}$ & \quad $[u_2,u_2]=0$ \\
 			$\{X_2,u_2,e_3\}$ & \quad $[v_2,u_2]=0$ \\
 			$\{X_2,v_2,e_3\}$ & \quad $[v_2,v_2]=0$ \\
 			$\{X_2,e_3,e_4\}$ & \quad $[v_2,e_4]=-d_{24}a_{33}^4X_4$ \\
 			$\{X_2,e_4,e_4\}$ & \quad $a_{44}^1=-d_{24}a_{44}^2$, \\
 			& \quad $\frac{1}{2}d_{24}a_{44}^3+\frac{1}{2}a_{24}a_{44}^2-a_{24}a_{33}^4-b_{24}d_{24}a_{33}^4=0$ \quad (1)\\
 			$\{e_4,e_4,e_4\}$ & \quad $a_{44}^3=b_{24}a_{44}^2$, \\
 			 & \quad  $d_{24}a_{44}^3+a_{24}a_{44}^2=0$ \quad (2)\\
 			
 			\end{tabular}
 \end{center}
From equation (2), (1) remains $a_{33}^4(a_{24}+b_{24}d_{24})=0$ and (2) can be rewritten as $a_{44}^2(a_{24}+b_{24}d_{24})=0$. After renaming $a_{44}^2:=a$, $a_{44}^4:=b$ and $a_{33}^4:=c$ we obtain as bracket products for the Lie superalgebra:

$$\begin{array}{ll}
    [X_1,X_2]=X_3, &  [X_1,e_4]=e_3\\{}
    [X_1,X_3]=X_4, &  [X_2,e_4]=d_{24}e_3+a_{24}u_2+b_{24}v_2\\{}
    [X_1,e_3]=u_2, & [X_3,e_4]=d_{24}u_2-v_2\\{}
    [X_2,e_3]=v_2,& [e_4,e_4]=-d_{24}aX_1+aX_2+ab_{24}X_3+bX_4 \\{} 
    [e_3,e_4]=\frac{1}{2}(aX_3+ab_{24}X_4),&[e_3,e_3]=cX_4\\{}
[u_2,e_4]=(\frac{1}{2}a-c)X_4,& [v_2,e_4]=-d_{24}cX_4
    \end{array}$$
with $a(a_{24}+b_{24}d_{24})=0$ and $c(a_{24}+b_{24}d_{24})=0$.

Now, we impose that this Lie superalgebra admits a bilinear form $B$ odd, supersymmetrical, invariant and non-degenerate. All the values of $B$ will came totally determined by 
$$\begin{array}{llll}
B(X_1,e_4),& B(X_1,e_3), & B(X_1,u_2),& B(X_1,v_2), \\
B(X_2,e_4),& B(X_2,e_3), & B(X_2,u_2),& B(X_2,v_2), \\
B(X_3,e_4),&B(X_3,e_3), & B(X_3,u_2),& B(X_3,v_2), \\
B(X_4,e_4),&B(X_4,e_3), & B(X_4,u_2),& B(X_4,v_2). \\
\end{array}$$

By applying the invariant condition $B([x,y],z)=B(x,[y,z])$
for the  ordered triple $\{x,y,z\}$  we get the  following relationships given
 in the table:
 \begin{center}
 	\begin{tabular}{l|l}
 		
 		Invariant Condition & \quad Relationship\\
 		\hline \hline
 		$\{X_1,X_1,e_4\}$ & \quad $B(X_1,e_3)=0$ \\
 			$\{X_1,X_1,e_3\}$ & \quad $B(X_1,u_2)=0$ \\
 			$\{X_2,X_2,e_3\}$ & \quad $B(X_2,v_2)=0$ \\
 			$\{X_4,X_1,e_4\}$ & \quad $B(X_4,e_3)=0$ \\
 			$\{X_3,X_2,e_3\}$ & \quad $B(X_3,v_2)=0$ \\
 			$\{X_4,X_2,e_3\}$ & \quad $B(X_4,v_2)=0$ \\
 			$\{X_1,X_2,u_2\}$ & \quad $B(X_3,u_2)=0$ \\
 			$\{X_1,X_3,u_2\}$ & \quad $B(X_4,u_2)=0$ \\
 			$\{X_1,X_2,e_3\}$ & \quad $B(X_3,e_3)=B(X_1,v_2)$ \\
 			$\{X_2,X_1,e_3\}$ & \quad $B(X_3,e_3)=-B(X_2,u_2)$ \\
 			$\{X_3,X_1,e_4\}$ & \quad $B(X_3,e_3)=-B(X_4,e_4)$ \\
 			$\{X_2,X_1,e_4\}$ & \quad $B(X_2,e_3)=-B(X_3,e_4)$ \\
 			$\{X_1,X_2,e_4\}$ & \quad $B(X_3,e_4)=b_{24}B(X_1,v_2)$ \\
 			\end{tabular}
 \end{center}
 Since $B$ is non-degenerate we get $B(X_1,v_2)=-B(X_2,u_2)\neq 0$ and from the invariant condition for the triples  $\{ X_2,X_3,e_4\}$ and $\{ X_2,e_4,X_2\}$ we get $d_{24}=0$ and $a_{24}=0$ respectively. Finally by the invariant condition on the triple  $\{ e_4,e_4,u_2\}$  we get $c=-\frac{1}{2}a$. Remaining then,   
 $$\begin{array}{ll}
    [X_1,X_2]=X_3, &  [X_1,e_4]=e_3\\{}
    [X_1,X_3]=X_4, &  [X_2,e_4]=b_{24}v_2\\{}
    [X_1,e_3]=u_2, & [X_3,e_4]=-v_2\\{}
    [X_2,e_3]=v_2,& [e_4,e_4]=aX_2+ab_{24}X_3+bX_4 \\{} 
    [e_3,e_4]=\frac{1}{2}(aX_3+ab_{24}X_4),&[e_3,e_3]=-\frac{1}{2}aX_4\\{}
[u_2,e_4]=aX_4,& 
    \end{array}$$
 $$\begin{array}{ll}
B(X_1,v_2)=\lambda, & B(X_1,e_4)=\alpha, \\
B(X_2,u_2)=-\lambda, & B(X_2,e_4)=\beta, \\
B(X_3,e_3)=\lambda, &  B(X_2,e_3)=-b_{24}\lambda,\\
B(X_4,e_4)=-\lambda, &  B(X_3,e_4)=b_{24}\lambda,\\
\end{array}$$
with $\lambda \neq 0$. After applying the isomorphism where all the basis vectors remain invariant except for 
$ \{ X'_2=X_2+b_{24}X_3, \ X'_3=X_3+b_{24}X_4\}$ we get $b_{24}=0$ and therefore $B(X_2,e_3)=B(X_3,e_4)=0$. 

Now, we distinguish two cases depending on $a\neq 0$ or $a=0$. If $a\neq 0$ the isomorphism (change of scale) defined by $$X'_i=X_i, \ 1 \leq i \leq 3, \ 
e'_4=\frac{1}{\sqrt{a}}e_4,
e'_3=\frac{1}{\sqrt{a}}e_3, \ u'_2=\frac{1}{\sqrt{a}}u_2, \ v'_2=\frac{1}{\sqrt{a}}v_2$$
allows to assert that $a$ can be supposed $a=1$ and renaming $X'_2:=X_2+bX_4$ (isomorphism) we get $b=0$ and then ${\mathfrak g}^2_8$. On the contrary, if $a=0$, we obtain  ${\mathfrak g}^0_8$ for $b=0$, and for $b\neq 0$ after the change of scale $\{X'_i=X_i, \ 1 \leq i \leq 3, \ 
e'_4=\frac{1}{\sqrt{b}}e_4,
e'_3=\frac{1}{\sqrt{b}}e_3, \ u'_2=\frac{1}{\sqrt{b}}u_2, \ v'_2=\frac{1}{\sqrt{b}}v_2\}$ we get $b=1$ and then ${\mathfrak g}^1_8$, which concludes the proof of the theorem. 
\end{proof}

\begin{cor}
The three non-isomorphic complex $8$-dimensional Lie superalgebras can be obtained by a double extension of the $6$-dimensional ${\mathfrak g}^0_6$. In fact, a straightforward computation leads to the expression for all the odd skew-supersymmetric derivations $D$ of ${\mathfrak g}^0_6$.  After replacing $\lambda$ by $1$ it can be obtained: 
$$\begin{array}{ll}
D(X_1)=a e_3+(a \beta +b \alpha +c)u_2-a\alpha v_2, & D(e_3)=fX_3\\
D(X_2)=b e_3+b \beta u_2+c v_2, & D(u_2)=D(v_2)=0\\
D(X_3)=bu_2-av_2&
\end{array}$$
      
Thus, using the notation of Proposition \ref{generalized} we have:
      
(I). ${\mathfrak g}^0_8$ comes from ${\mathfrak g}^0_6$ with: \begin{itemize}
\item[$-$] $D(X_1)=-e_3, \ D(X_3)=v_2$. This derivation derives from the general expression for $a=-1$ and null the others parameters. 
\item[$-$] $\varphi(X_1,X_3)=1$.
\item[$-$]$[e_4,e_4]=0$.
\end{itemize}

(II). ${\mathfrak g}^1_8$ comes from ${\mathfrak g}^0_6$ with: \begin{itemize}
\item[$-$] $D(X_1)=-e_3, \ D(X_3)=v_2$. 
\item[$-$] $\varphi(X_1,X_3)=1$.
\item[$-$]$[e_4,e_4]=X_4$.
\end{itemize}

(III). ${\mathfrak g}^2_8$ comes from ${\mathfrak g}^0_6$ with: \begin{itemize}
\item[$-$] $D(X_1)=-e_3, \ D(X_3)=v_2, \ D(e_3)=\frac{1}{2}X_3$. This derivation derives from the general expression for $a=-1$, $f=\frac{1}{2}$ and null the others parameters. 
\item[$-$] $\varphi(X_1,X_3)=1$, $\varphi(e_3,e_3)=-\frac{1}{2}$.
\item[$-$]  $\psi(u_2)=1$.
\item[$-$]$[e_4,e_4]=X_2$.
\end{itemize}
\end{cor}

A possible future line of research is to study the class of quadratic Lie superalgebras $(\mathfrak{g}=\mathfrak{g}_{\bar 0}\oplus \mathfrak{g}_{\bar 1},[\cdot,\cdot],B)$ such that $\mathfrak{g}_{\bar 1}$ is a weak filiform $\mathfrak{g}_{\bar 0}$-module. Additionally, we have in mind   another possible  modification of  the definition of filiform and study the set of quadratic (and odd-quadratic) Lie superalgebra $(\mathfrak{g}=\mathfrak{g}_{\bar 0}\oplus \mathfrak{g}_{\bar 1},[\cdot,\cdot],B)$ such that the $\mathfrak{g}_{\bar 0}$-module $\mathfrak{g}_{\bar 1}$ is a filiform module of this new type.

\bibliographystyle{amsplain}

\end{document}